\documentclass[a4paper,11pt]{amsart}
\usepackage[left=2.7cm,right=2.7cm,top=3.5cm,bottom=3cm]{geometry}

\usepackage{amssymb,latexsym,amsmath,amsthm,amscd}
\usepackage{stmaryrd}

\usepackage{graphicx}
\usepackage{dsfont}
\usepackage[all]{xy}
\usepackage{mathrsfs}
\usepackage{color}
\definecolor{Blue}{rgb}{0.3,0.3,0.9}
\usepackage{array}
\usepackage{hyperref}
\usepackage{tikz-cd}

\usepackage[T2A,OT1]{fontenc}
\DeclareSymbolFont{cyrillic}{T2A}{cmr}{m}{n}
\DeclareMathSymbol{\Sha}{\mathalpha}{cyrillic}{216}

\newcommand{\sk}{\vspace{0.1in}}
\newtheorem{thm}{Theorem}[section]
\newtheorem{pro-thm}[thm]{Main Theorem}
\newtheorem{def-thm}[thm]{Definition-Theorem}
\newtheorem{cor}[thm]{Corollary}
\newtheorem{lem}[thm]{Lemma}

\newtheorem{prop}[thm]{Proposition}

\newtheorem{conj}[thm]{Conjecture}

\newtheorem*{conjA'}{Conjecture~A$'$}

\newtheorem*{assumption-spade}{Hypothesis~$\spadesuit$}
\theoremstyle{definition}
\newtheorem{defn}[thm]{Definition}
\theoremstyle{remark}
\newtheorem{rem}[thm]{Remark}

\newtheorem{intro-rem}{Remark}
\numberwithin{equation}{section}

\newcommand{\abs}[1]{\left\vert#1\right\vert}

\newcommand{\rH}{{\rm H}}

\newcommand{\pp}{\mathfrak{p}}
\newcommand{\ppbar}{\overline{\mathfrak{p}}}
\newcommand{\qq}{\mathfrak{q}}
\newcommand{\qqbar}{\overline{\mathfrak{q}}}

\newcommand{\bZ}{\mathbf{Z}}

\def\cK{{\mathcal K}}  

\def\cO{\mathcal O}

\def\cP{{\mathcal P}}

\newcommand{\Q}{\mathbf{Q}}
\newcommand{\Z}{\mathbf{Z}}

\def\makeop#1{\expandafter\def\csname#1\endcsname
	{\mathop{\rm #1}\nolimits}\ignorespaces}
\makeop{Hom}   \makeop{End}   \makeop{Aut}   
\makeop{Pic} \makeop{Gal}       \makeop{Div} \makeop{Lie}
\makeop{PGL}   \makeop{Corr} \makeop{PSL} \makeop{sgn} \makeop{Spf}
\makeop{Tr} \makeop{Nr} \makeop{Fr} \makeop{disc}
\makeop{Proj} \makeop{supp} \makeop{ker}   \makeop{Im} \makeop{dom}
\makeop{coker} \makeop{Stab} \makeop{SO} \makeop{SL} \makeop{SL}
\makeop{Cl}    \makeop{cond} \makeop{Br} \makeop{inv} \makeop{rank}
\makeop{id}    \makeop{Fil} \makeop{Frac}  \makeop{GL} \makeop{SU}
\makeop{Trd}   \makeop{Sp} \makeop{Tr}    \makeop{Trd} \makeop{Res}
\makeop{ind} \makeop{depth} \makeop{Tr} \makeop{st} \makeop{Ad}
\makeop{Int} \makeop{tr}    \makeop{Sym} \makeop{can} \makeop{SO}
\makeop{torsion} \makeop{GSp} \makeop{Tor}\makeop{Ker} \makeop{rec}
\makeop{Ind} \makeop{Coker}
\makeop{vol} \makeop{Ext} \makeop{gr} \makeop{ad}
\makeop{Gr}\makeop{corank} \makeop{Ann}
\makeop{Hol} 
\makeop{Fitt} \makeop{Mp} \makeop{CAP}

\newcommand{\beqcd}[1]{\begin{equation*}\label{#1}\tag{#1}}
\newcommand{\eeqcd}{\end{equation*}}

\newcommand{\dBr}[1]{\llbracket{#1}\rrbracket}

\begin{document}
\title[The leading coefficient of the BDP $p$-adic $L$-function]{Derived $p$-adic heights and the leading coefficient of the Bertolini--Darmon--Prasanna $p$-adic $L$-function}

\author[F.~Castella]{Francesc Castella}
\address[Castella]{Department of Mathematics, University of California Santa Barbara, CA 93106}
\email{castella@ucsb.edu}

\author[C-Y.~Hsu]{Chi-Yun Hsu}
\address[Hsu]{Department of Mathematics and Computer Science, Santa Clara University, Santa Clara, CA 95053, USA}
\email{chsu6@scu.edu}

\author[D.~Kundu]{Debanjana Kundu}
\address[Kundu]{Department of Mathematical and Statistical Sciences\\ UTRGV \\ 1201 W University Dr.\\ Edinburg, TX 78539\\ USA}
\email{dkundu@math.toronto.edu}

\author[Y-S.~Lee]{Yu-Sheng Lee}
\address[Lee]{Department of Mathematics, University  of Michigan, Ann Arbor, MI 48109, USA}
\email{yushglee@umich.edu}

\author[Z.~Liu]{Zheng Liu}
\address[Liu]{Department of Mathematics, University of California Santa Barbara, CA 93106, USA}
\email{zliu@math.ucsb.edu}

\date{\today}


\begin{abstract}
Let $E/\Q$ be an elliptic curve and let $p$ be an odd prime of good reduction for $E$. Let $K$ be an imaginary quadratic field satisfying the classical Heegner hypothesis and in which $p$ splits. The goal of this paper is two-fold:
\begin{enumerate}
\item We formulate a $p$-adic BSD conjecture for the $p$-adic $L$-function $L_\pp^{\rm BDP}$ introduced by Bertolini--Darmon--Prasanna \cite{bdp1}. 
\item For an algebraic analogue $F_{\ppbar}^{\rm BDP}$ of $L_\pp^{\rm BDP}$, we show that the ``leading coefficient'' part of our conjecture holds, and that the ``order of vanishing'' part follows from the expected ``maximal non-degeneracy'' of an anticyclotomic $p$-adic height.  
\end{enumerate}
In particular, when the Iwasawa--Greenberg Main Conjecture   
$(F_{\ppbar}^{\rm BDP})=(L_\pp^{\rm BDP})$ is known, our results determine the leading coefficient of $L_{\pp}^{\rm BDP}$ at $T=0$ up to a $p$-adic unit.  Moreover, by adapting the approach of Burungale--Castella--Kim \cite{BCK}, 
we prove the main conjecture for supersingular primes $p$ under mild hypotheses. 

In the $p$-ordinary case, and under some additional hypotheses, similar results were obtained by Agboola--Castella \cite{AC}, but our method is new and completely independent from theirs, and apply to all good primes.

\end{abstract}

\maketitle
\setcounter{tocdepth}{1}
\tableofcontents

\section{Introduction}

\subsection{The BDP \texorpdfstring{$p$}{}-adic \texorpdfstring{$L$}{}-function}
Let $E/\Q$ be an elliptic curve of conductor $N$ and let $p$ be an odd prime of good reduction for $E$.
Set $f\in S_2(\Gamma_0(N))$ to denote the newform associated with $E$.  Let $K$ be an imaginary quadratic field of discriminant prime to $Np$, and assume the classical \emph{Heegner hypothesis}, i.e., that
\begin{equation}\label{eq:Heeg}
\textrm{every prime factor of $N$ splits in $K$.}\tag{Heeg}
\end{equation}
Fix an embedding $\iota_p:\overline{\Q}\hookrightarrow\overline{\Q}_p$, and  assume also that
\begin{equation}\label{eq:spl}
\textrm{$p=\pp\ppbar$ splits in $K$,}\tag{spl}
\end{equation} 
with $\pp$ the prime of $K$ above $p$ induced by $\iota_p$.
Let $K_\infty/K$ be the anticyclotomic $\Z_p$-extension, and put
\[
\Gamma={\rm Gal}(K_\infty/K),\qquad
\Lambda=\Z_p\dBr{\Gamma},\qquad\Lambda_{\widehat{\cO}}=\Lambda\widehat\otimes_{\Z_p}\widehat{\cO},
\]
where  $\widehat{\cO}$ is the completion of the ring of integers of the maximal unramified extension of $\Q_p$.

In a seminal paper \cite{bdp1}, Bertolini--Darmon--Prasanna introduced a $p$-adic $L$-function
\[
\mathscr{L}_\pp^{\rm BDP}\in\Lambda_{\widehat{\cO}}
\]
whose square $L_\pp^{\rm BDP}=(\mathscr{L}_\pp^{\rm BDP})^2$ interpolates central critical values of the complex $L$-function of $f/K$ twisted by  infinite order characters of $\Gamma$.
The main result in \emph{op.\,cit.}, asserts that the value of $L_\pp^{\rm BDP}$ at the trivial character $\mathds{1}$ of $\Gamma$ (which lies outside the range of interpolation) is given by
\begin{equation}\label{eq:bdp}
L_\pp^{\rm BDP}(\mathds{1})=\frac{1}{u_K^{2}c_E^{2}}\cdot\biggl(\frac{1-a_p(E)+p}{p}\biggr)^2\cdot\log_{\omega_E}(z_K)^2.\tag{BDP}
\end{equation}

Here, $a_p(E):=p+1-\#E(\mathbf{F}_p)$, $u_K:=\frac{1}{2}\#\cO_K^\times$, $z_K\in E(K)$ is a Heegner point arising from a modular parametrisation 
\[
\varphi_E:X_0(N)\rightarrow E
\] 
with associated Manin constant\footnote{using \cite{mazur-prime} for the inclusion $c_E\in\Z_{(p)}$.} $c_E\in\Z_{>0}\cap\Z_{(p)}$ (so $z_K\otimes c_E^{-1}\in E(K)\otimes\Z_p$ is independent of the choice of  $\varphi_E$), and 
\[
\log_{\omega_E}:E(K_\pp)\otimes\bZ_p\rightarrow\bZ_p
\] 
is the formal group logarithm associated with a N\'{e}ron differential $\omega_E\in\Omega^1(E/\bZ_{(p)})$.

The above formula \eqref{eq:bdp} has been a key ingredient in recent progress over the past decade towards the Birch--Swinnerton-Dyer conjecture when the analytic rank of $E$ is $\leq 1$: \cite{JSW}, \cite{skinner}, etc. 
(see \cite{bur-BDP} and the references therein).


The goal of this paper is to formulate and study a $p$-adic analogue of the Birch--Swinnerton-Dyer conjecture for $L_\pp^{\rm BDP}$ for \emph{all good primes} $p>2$, predicting: 
\begin{itemize}
\item[(i)] the ``order of vanishing'' of $L_\pp^{\rm BDP}$ at the trivial character $\mathds{1}$;  
\item[(ii)] a formula for the ``leading coefficient'' of $L_\pp^{\rm BDP}$ at $\mathds{1}$.
\end{itemize} 
%


In the $p$-ordinary case, this task was first carried out by Agboola--Castella \cite{AC} drawing from the methods of Bertolini--Darmon \cite{BD-derived-AJM}.
The formulation of the conjecture in \cite{AC} imposed some technical hypotheses required for the existence of a ``perfect control theorem'': $p\nmid c_\ell$ for the Tamagawa numbers $c_\ell$ of $E/\Q_\ell$ for all primes $\ell\mid N$, and $p\nmid\#E(\mathbf{F}_p)$. Such control theorem is well-known to fail in the supersingular case. 
Moreover, it was also assumed that $\#\Sha(E/K)[p^\infty]<\infty$.
 
The new approach in this paper allows us to give a formulation of a $p$-adic BSD conjecture for $L_\pp^{\rm BDP}$ without any of those additional hypotheses and for all good primes $p>2$.
Moreover, modulo the expected ``maximal non-degeneracy'' of an anticyclotomic $p$-adic height pairing, we prove our conjecture for an algebraic analogue of $L_\pp^{\rm BDP}$.

\subsection{\texorpdfstring{$p$-adic analogue of BSD for $L_\pp^{\rm BDP}$}{}}

For the formulation of our conjecture, we assume that the triple $(E,K,p)$ satisfies the following additional hypotheses:
\begin{equation}\label{eq:h0-intro}
E(K)[p]=0, \tag{h0}
\end{equation}
and that for every $\qq\in\{\pp,\ppbar\}$ the restriction map
\begin{equation}\label{eq:h1-intro}
{\rm res}_\qq:\check{S}_p(E/K)\rightarrow E(K_\qq)\otimes\Z_p\tag{h1}
\end{equation}
has non-torsion image, where $\check{S}_p(E/K)=\varprojlim_k{\rm Sel}_{p^k}(E/K)$ is the usual compact Selmer group. (Note that (\ref{eq:h1-intro}) is implied by the finiteness of $\Sha(E/K)[p^\infty]$, since by the $p$-parity conjecture \cite{nekovarII,kim-parity}, hypothesis (\ref{eq:Heeg}) implies that the Selmer group $\check{S}_p(E/K)$ has odd $\Z_p$-rank.)


Denote by $E(K_\qq)_{/{\rm tor}}$ the quotient of $E(K_\qq)$ by its torsion submodule, and let ${\rm res}_{\qq/{\rm tor}}$ be the composition of ${\rm res}_\qq$ with the projection $E(K_\qq)\otimes\Z_p\rightarrow E(K_\qq)_{/{\rm tor}}\otimes\Z_p$.
Let 
\[
T=\varprojlim_k E[p^k]
\]
be the $p$-adic Tate module of $E$, 
and set ${\rm Sel}_\qq(K,T)={\rm ker}({\rm res}_{\qq/{\rm tor}})$.

In Section~\ref{sec:der-ht}, building on Howard's theory of derived $p$-adic heights \cite{howard-derived}, we construct a filtration 
\[
{\rm Sel}_\qq(K,T)\supset
\underleftarrow{\mathfrak{S}}_\qq^{(1)}\supset\underleftarrow{\mathfrak{S}}_\qq^{(2)}\supset\cdots\supset\underleftarrow{\mathfrak{S}}_\qq^{(i)}\supset\cdots,
\]
with $\underleftarrow{\mathfrak{S}}_\pp^{(i)}=\{0\}$ for $i\gg 0$, 
equipped with a sequence of ``derived'' $p$-adic height pairings
\begin{equation}
\label{eq:intro-hpi}
h_\pp^{(i)}:\underleftarrow{\mathfrak{S}}_\pp^{(i)}\times\underleftarrow{\mathfrak{S}}_{\ppbar}^{(i)}\rightarrow J^i/J^{i+1},
\end{equation}
where $J\subset\Lambda$ is the augmentation ideal.
%
Using these pairings, 
we define a 
$p$-adic regulator 
\[
{\rm Reg}_{\pp,{\rm der}}\in\bigl((J^\sigma/J^{\sigma+1})\otimes_{\Z_p}\Q_p\bigr)/\Z_p^\times,
\] 
where $\sigma=\sum_{i\geq 1}i\cdot{\rm rank}_{\Z_p}(\underleftarrow{\mathfrak{S}}_\pp^{(i)}/\underleftarrow{\mathfrak{S}}_\pp^{(i+1)})$. By construction, ${\rm Reg}_{\pp,{\rm der}}$ is \emph{always} nonzero.

Set
\[
r:={\rm rank}_{\Z_p}\check{S}_p(E/K).
\]
Under hypotheses \eqref{eq:h0-intro}-\eqref{eq:h1-intro}, the Selmer groups ${\rm Sel}_\qq(K,T)\subset\check{S}_p(E/K)$ are free $\Z_p$-modules of rank $r-1$ and $r$, respectively.
Let $(s_1,\dots,s_{r-1})$ be a $\Z_p$-basis for ${\rm Sel}_\pp(K,T)$, and extend it to a $\Z_p$-basis $(s_1,\dots,s_{r-1},s_\pp)$ for $\check{S}_p(E/K)$.
In particular, ${\rm res}_{\pp/{\rm tor}}(s_\pp)\neq 0$.


The following is our $p$-adic BSD conjecture for $L_\pp^{\rm BDP}$.

\begin{conj}[$p$-adic BSD conjecture for $L_\pp^{\rm BDP}$]\label{conj:BSD-intro}
Assume {\rm \eqref{eq:h0-intro}}-{\rm \eqref{eq:h1-intro}}.
\begin{itemize}
\item[(i)] {\rm (Leading Coefficient Formula)} Let 
\[
\varrho_{\rm an}={\rm ord}_JL_\pp^{\rm BDP}:=
\max\{i\geq 0\,\colon\,L_\pp^{\rm BDP}\in J^i\},
\] 
and denote by $\overline{L_\pp^{\rm BDP}}$ the natural image of $L_\pp^{\rm BDP}$ in $J^{\varrho_{\rm an}}/J^{\varrho_{\rm an}+1}$.
Then, up to a $p$-adic unit
\[
\overline{L_\pp^{\rm BDP}}=\biggl(\frac{1-a_p(E)+p}{p}\biggr)^2\cdot{\rm log}_\pp(s_\pp)^2\cdot{\rm Reg}_{\pp,{\rm der}}\cdot\#\Sha_{\rm BK}(K,W)\cdot\prod_{\ell\mid N}c_\ell^2.
\]
\item[(ii)] {\rm (Order of Vanishing)}
Let $r^\pm$ denote the $\Z_p$-rank of the $\pm$-eigenspace of $\check{S}_p(E/K)$ under the action of complex conjugation.
Then,
\[
\varrho_{\rm an}=2(\max\{r^+,r^-\}-1).
\]
\end{itemize}
\end{conj}

Here $c_\ell$ is the Tamagawa number of $E/\Q_\ell$, $W=E[p^\infty]$, and $\Sha_{\rm BK}(K,W)={\rm Sel}_{p^\infty}(E/K)_{/{\rm div}}$ 
is the Bloch--Kato Tate--Shafarevich group, i.e., the quotient of the $p^\infty$-Selmer group ${\rm Sel}_{p^\infty}(E/K)$ by its maximal divisible submodule.
Also, 
\[
{\rm log}_\pp:\check{S}_p(E/K)\rightarrow\Z_p
\] 
denotes the composition ${\rm log}_{\omega_E}\circ{\rm res}_{\pp/{\rm tor}}$.

\begin{rem}\label{rem:cf-AC}
In the $p$-ordinary case, a variant of Conjecture~\ref{conj:BSD-intro} was formulated in \cite{AC}, 
with a regulator defined using the theory of derived $p$-adic heights by Bertolini--Darmon \cite{BD-derived-AJM}. An important advantage of the present formulation in comparison with the formulation of \cite[Conj.~4.2]{AC} is that the latter requires a hypothesis on the derived $p$-adic heights  amounting to the requirement that $\underleftarrow{\mathfrak{S}}_\qq^{(i)}=0$ for $i\geq p$ 
for the definition of their derived $p$-adic regulator in [\emph{op.\,cit.}, Def.~4.1], whereas our formulation  of Conjecture~\ref{conj:BSD-intro} is completely unconditional.
\end{rem}

\begin{rem}\label{rem:motivic}
Assuming $\#\Sha(E/K)[p^\infty]<\infty$ for the usual Tate--Shafarevich group $\Sha(E/K)$, it is possible to remove the ambiguity by a $p$-adic unit in the formulation of Conjecture~\ref{conj:BSD-intro}(i).
Indeed, following an observation from \cite[Rem.~2.21]{BD-derived-AJM}, this can be achieved as follows: in this case 
\[
\check{S}_p(E/K)\simeq E(K)\otimes\Z_p=(E(K)/E(K)_{\rm tor})\otimes\Z_p\simeq\Z_p^r,
\]
using \eqref{eq:h0-intro} for the middle equality.
Let $M$ be an endomorphism of $\check{S}_p(E/K)$ sending a $\Z$-basis $(P_1,\dots,P_r)$ of $E(K)/E(K)_{\rm tor}\simeq\Z^r$ to $(s_1,\dots,s_{r-1},s_\pp)$.
Then it suffices to replace ${\rm Reg}_{\pp,{\rm der}}$ in the right-hand side of Conjecture~\ref{conj:BSD-intro}(i) by the modification
\[
{\rm det}(M)^{-2}\cdot{\rm Reg}_{\pp,{\rm der}},
\]
which is a well-defined element in $(J^\sigma/J^{\sigma+1})\otimes_{\Z_p}\Q_p$ 
and is independent of the choice of $M$.
\end{rem}

\begin{rem}
When ${\rm ord}_{s=1}L(E/K,s)=1$,  Conjecture~\ref{conj:BSD-intro}(ii) follows immediately from \eqref{eq:bdp} and the work of Gross--Zagier and Kolyvagin.
In this case, $\varrho_{\rm an}=0$, and  the Leading Coefficient Formula in Conjecture~\ref{conj:BSD-intro}(i) is \emph{equivalent} to the $p$-part of the Birch--Swinnerton-Dyer formula for $L'(E/K,1)$ (see Proposition~\ref{prop:rank-1}).
\end{rem}


\subsection{Main results}

By the Iwasawa--Greenberg Main Conjecture (see Conjecture~\ref{conj:BDP-IMC}), 
$L_\pp^{\rm BDP}$ should generate the characteristic ideal of a $\Lambda$-adic Selmer group denoted $X_{\ppbar}={\rm Sel}_{\ppbar}(K_\infty,W)^\vee$ 
in the body of the paper.
This module is known to be $\Lambda$-torsion under hypothesis \eqref{eq:h0-intro} (see Proposition~\ref{prop:Lambda-tors}).

The main result of this paper is the following.

\begin{thm}\label{thm:intro-A}
Let $F_{\ppbar}^{\rm BDP}\in\Lambda$ be a generator of ${\rm char}_\Lambda(X_{\ppbar})$, and put 
\[
\varrho_{\rm alg}={\rm ord}_{J}F_{\ppbar}^{\rm BDP}:=
\max\{i\geq 0\,\colon\,F_{\ppbar}^{\rm BDP}\in J^i\}.
\]
\begin{itemize} 
\item[(i)] Let $\overline{F_{\ppbar}^{\rm BDP}}$ be the natural image of $F_{\ppbar}^{\rm BDP}$ in $J^{\varrho_{\rm alg}}/J^{{\varrho_{\rm alg}}+1}$.
Then, up to a $p$-adic unit
\[
\overline{F_{\ppbar}^{\rm BDP}}=\biggl(\frac{1-a_p(E)+p}{p}\biggr)^2\cdot{\rm log}_{\pp}(s_\pp)^2\cdot{\rm Reg}_{\pp,{\rm der}}\cdot\#\Sha_{\rm BK}(K,W)\cdot\prod_{\ell\mid N}c_\ell^2.
\]
\item[(ii)] Furthermore,
\[
{\varrho_{\rm alg}}\geq 2(\max\{r^+,r^-\}-1),
\]
with equality if and only if ${\rank}_{\Z_p}\underleftarrow{\mathfrak{S}}_\pp^{(2)}=\abs{ r^+-r^-}-1$ and $\underleftarrow{\mathfrak{S}}_\pp^{(i)}=0$ for $i\geq 3$.
\end{itemize}
\end{thm}

Set $h_\pp:=h_\pp^{(1)}$.
We say that  $h_\pp$ is \emph{maximally non-degenerate} when the condition for equality in the last part of Theorem~\ref{thm:intro-A} holds.
In the $p$-ordinary case, conjectures due to Mazur and Bertolini--Darmon (\cite[\S{3}]{BD-derived-AJM}) 
imply that $h_\pp$ is maximally non-degenerate (see Remark~\ref{rem:max-nondeg}), and we expect this condition to hold for all good primes $p>2$.
(For the second assertion in Theorem~\ref{thm:intro-A}(ii), note that hypothesis \eqref{eq:Heeg} and the $p$-parity conjecture imply that $r^+$ and $r^-$ have different parities, and so $\abs{r^+-r^-}-1\geq 0$.)

As a consequence, when the Iwasawa--Greenberg Main Conjecture 
\begin{equation}\label{eq:BDP-IMC-intro}
\bigl(F_{\ppbar}^{\rm BDP}\bigr)\overset{?}=\bigl(L_\pp^{\rm BDP}\bigr)
\end{equation}
is known, Theorem~\ref{thm:intro-A} implies Conjecture~\ref{conj:BSD-intro}(i) (up to a $p$-adic unit). For ordinary primes, a proof of the Main Conjecture \eqref{eq:BDP-IMC-intro} was obtained in \cite{BCK}. In Section~\ref{sec:IMC} we extend their methods to the supersingular case, leading to the following result.

\begin{thm}\label{thm:intro-B}
Let $p>2$ be a prime of good supersingular reduction for $E$, and let $K$ be an imaginary quadratic field satisfying \eqref{eq:Heeg} and \eqref{eq:spl}. Assume in addition that:
\begin{itemize}
\item[(i)] $N$ is square-free.
\item[(ii)] $E[p]$ is ramified at every prime $\ell\mid N^+$. 
\item[(iii)] $E[p]$ is ramified at every prime $\ell\mid N^-$ with $\ell\equiv\pm{1}\pmod{p}$.
\item[(iv)] Every prime above $p$ is totally ramified in $K_\infty/K$.
\end{itemize}
Then $X_{\ppbar}$ is $\Lambda$-torsion, with
\[
{\rm char}_\Lambda\bigl( X_{\ppbar}\bigr)=\bigl(L_\pp^{\rm BDP}\bigr).
\]
In other words, the Iwasawa--Greenberg Main Conjecture \eqref{eq:BDP-IMC-intro} holds.
\end{thm}

%

Hence (together with the proof of \eqref{eq:BDP-IMC-intro} in \cite{BCK} in the $p$-ordinary case, since Theorem~\ref{thm:intro-A} also applies to this case), our results in this paper determine 
(up to a $p$-adic unit) the leading coefficient of $L_{\pp}^{\rm BDP}$ at $T=0$ for all primes $p>2$ of good reduction, and reduce a full proof of Conjecture~\ref{conj:BSD-intro} to the expected maximal non-degeneracy of $h_\pp$. 

\subsection{Relation to prior work}
In the $p$-ordinary case, a version of Theorem~\ref{thm:intro-A} was obtained in \cite[Thm.~6.12]{AC} under some additional hypotheses, including: 
\begin{itemize}
\item[(i)] $a_p(E)\not\equiv 1\pmod{p}$, 
\item[(ii)] $p\nmid c_\ell$ for all primes $\ell\mid N$,
\item[(iii)] $\#\Sha(E/K)[p^\infty]<\infty$.
\end{itemize}
As noted in [\emph{op.\,cit.}, \S{6.2}], their proof is largely based on an adaptation of the arguments from \cite{BD-derived-AJM}, and hypotheses (i) and (ii) are essential to their method. Our proof, building on the approach to derived $p$-adic heights later developed by Howard \cite{howard-derived}, is completely different from and independent of \cite{AC}, and we are able to dispense with hypotheses (i)--(iii) (and obtain a result also in the $p$-supersingular case).

Back to the $p$-ordinary case, Sano \cite{sano} has also given a new proof of \cite[Thm.~6.12]{AC} removing hypotheses (i) and (ii) 
by building on an extension of Nekov{\'a}{\v{r}}'s descent formalism \cite[\S{11.6}]{nekovar310} using ``derived'' Bockstein maps; his methods are also different from ours. 

Finally we note that for supersingular primes $p>2$, the results of Theorem~\ref{thm:intro-A} is new for all cases $\varrho_{\rm alg}>0$; on the other hand, the case $\varrho_{\rm alg}=0$  of  Theorem~\ref{thm:intro-A} (in which case ${\rm Sel}_\qq(K,T)=0$ and $p$-adic heights make no appearance)  
can be deduced from the ``anticyclotomic control theorem'' of Jetchev--Skinner--Wan \cite{JSW}, 
but our results also give a new proof in this case.

\subsection{Method of proof}

In a series of papers culminating in \cite{PR-ht} 
(see also \cite{colmez-Lp}), Perrin-Riou developed a method for computing the leading coefficient of algebraic $p$-adic $L$-functions in terms of arithmetic invariants. 
In the setting of this paper, her methods allow one to determine a formula for $\varrho_{\rm alg}$ and the image $\overline{F_{\ppbar}^{\rm BDP}}\in J^{\varrho_{\rm alg}}/J^{{\varrho_{\rm alg}}+1}$ (up to a $p$-adic unit)  precisely when the sequence
\begin{equation}\label{eq:4-term}
0\rightarrow\underleftarrow{\mathfrak{S}}_\pp^{(\infty)}\otimes\Q_p\rightarrow{\rm Sel}_\pp(K,V)\xrightarrow{h_\pp}{\rm Hom}({\rm Sel}_{\ppbar}(K,V),\Q_p)\rightarrow{\rm Hom}\bigl(\underleftarrow{\mathfrak{S}}_{\ppbar}^{(\infty)}\otimes\Q_p,\Q_p\bigr)\rightarrow 0\nonumber
\end{equation}
is exact, where $V=T\otimes\Q_p$ is the rational $p$-adic Tate module, $\underleftarrow{\mathfrak{S}}_\qq^{(\infty)}\otimes\Q_p=\cap_{i\geq 1}\underleftarrow{\mathfrak{S}}_\qq^{(i)}\otimes\Q_p$ is the space of universal norms in ${\rm Sel}_\qq(K,V)$, 
and the middle arrow is defined by 
the $p$-adic height pairing $h_\pp$ (see \cite[Thm.~{3.3.4}]{PR-ht}).
However, in the setting of this paper one can show that 
\[
\underleftarrow{\mathfrak{S}}_\qq^{(\infty)}=0,
\] 
(see  Proposition~\ref{prop:Lambda-tors} and Corollary~\ref{cor:der-ht}(ii)), and therefore exactness of \eqref{eq:4-term} amounts to the \emph{non-degeneracy} of $h_\pp$, which in the anticyclotomic setting fails as long as $\abs{ r^+-r^-}>1$ (see Proposition~\ref{prop:e2}). 

The key technical innovation of this paper is the development of an extension of the results of \cite{PR-ht} 
applicable to $X_{\ppbar}$, 
building on Howard's theory of derived $p$-adic heights \cite{howard-derived} to account for the systematic degeneracies of $h_\pp$ in the anticyclotomic setting. 

\subsection{Outline of the paper}  
In Section~\ref{sec:Sel} we introduce some of our Selmer groups of interest, and state the Iwasawa--Greenberg Main Conjecture \eqref{eq:BDP-IMC-intro}.
In Section~\ref{sec:der-ht}, we extend the results we need from \cite{howard-derived} to our setting, yielding the groundwork for the definition of the derived regulator ${\rm Reg}_{\pp,{\rm der}}$ in the formulation of our conjectures in Section~\ref{sec:BSDconj}.
In Section~\ref{sec:main}, we state our main results toward the $p$-adic Birch--Swinnerton-Dyer conjectures for $L_\pp^{\rm BDP}$, and Sections~\ref{sec:proofs}-\ref{sec:IMC} are devoted to the proofs. In particular, we refer the reader to the start of Section~\ref{sec:proofs} for an outline of the proof of our main Theorem~\ref{thm:intro-A}.




\subsection{Acknowledgements}

This work began at the APAW Collaborative Research Workshop held at the University of Oregon in the summer of 2022.
The authors would like to thank the organisers of this workshop, and especially Ellen Eischen, for creating such a stimulating and thought-provoking event.
We would also like to thank Takamichi Sano for sending us a copy of his recent preprint and for a number of helpful exchanges, and the anonymous referee for some useful comments and suggestions.

During the preparation of this paper, F.C. was partially supported by the NSF grant 
DMS-2101458, C.-Y.H. was partially supported by Agence Nationale de Recherche ANR-18-CE40-0029 in the project \textit{Galois representations, automorphic forms and their L-functions (GALF)}, D.K. was partially supported by a PIMS postdoctoral fellowship, and Z.L. was partially supported by the NSF grant DMS-2001527.

This work was partially supported by Ellen Eischen's NSF CAREER grant DMS-1751281, and by the NSF grant DMS-1928930 while some of the authors (F.C., Y.-S.L., and Z.L.) were in residence at MSRI/SLMath during the spring of 2023 
for the program ``Algebraic Cycles, $L$-Values, and Euler Systems''.

\section{Selmer groups}\label{sec:Sel}

We keep the notation from the Introduction, and assume that the triple $(E,K,p)$ satisfies hypotheses \eqref{eq:Heeg} and \eqref{eq:spl}.
For every $n>0$, we write $K_n$ for the subextension of the anticyclotomic $\Z_p$-extension $K_\infty/K$ with ${\rm Gal}(K_n/K)\simeq\bZ/p^n\bZ$.

Let $\Sigma$ be a finite set of places of $K$ containing the archimedean place $\infty$ and the primes dividing $Np$.
Denote by $\Sigma_f$ the set of finite places in $\Sigma$, and assume that all primes in $\Sigma_f$ split in $K$.
For every number field $F$ containing $K$, let $G_{F,\Sigma}={\rm Gal}(F^\Sigma/F)$ be the Galois group of the maximal algebraic extension of $F$ unramified outside the places above $\Sigma$.
Recall that $T$ denotes the $p$-adic Tate module of $E$, and put
\[
V=T\otimes_{\Z_p}\Q_p,\qquad W=V/T\simeq E[p^\infty].
\]


\begin{defn}\label{def:BDP-Sel}
Suppose $F\supset K$ is a field extension, and let $\qq\in\{\pp,\ppbar\}$ be any of the primes of $K$ above $p$.
We define the \emph{$\qq$-strict Selmer group} of $V$ by
\begin{equation}\label{eq:def-Sel}
{\rm Sel}_\qq(F,V):={\rm ker}\biggl\{{\rm H}^1(G_{F,\Sigma},V)\rightarrow\prod_{w\in\Sigma_f}\frac{{\rm H}^1(F_w,V)}{{\rm H}^1_{\qq}(F_w,V)}\biggr\},
\end{equation}
where
\[
{\rm H}^1_\qq(F_w,V)=\begin{cases}
{\rm H}^1(F_{w},V)&\textrm{if $w\mid\overline{\qq}$},\\
0&\textrm{else.}
\end{cases}
\]
Similarly, letting $\rH^1_{\rm str}(F_w,V)$ and $\rH^1_{\rm rel}(F_w,V)$ be defined by
\begin{align*}
\rH^1_{\rm str}(F_w,V)&=0\quad\textrm{for all $w\in\Sigma_f$,}
\end{align*}
and
\begin{align*}
{\rm H}^1_{\rm rel}(F_w,V)&=\begin{cases}
{\rm H}^1(F_{w},V)&\textrm{if $w\mid p$},\\
0&\textrm{else,}
\end{cases}
\end{align*}
we define the Selmer groups ${\rm Sel}_{\rm str}(F,V)$ and ${\rm Sel}_{\rm rel}(F,V)$ by replacing $\rH^1_\qq(F_w,V)$ by $\rH^1_{\rm str}(F_w,V)$ and $\rH^1_{\rm rel}(F_w,V)$, respectively, in \eqref{eq:def-Sel}.
\end{defn}

For $?\in\{\qq,{\rm str},{\rm rel}\}$, let $\rH^1_{?}(F_w,T)$ and $\rH^1_{?}(F_w,W)$ be the pre-image and image, respectively, of $\rH^1_{?}(F_w,V)$ under the natural exact sequence
\[
\rH^1(F_w,T)\rightarrow\rH^1(F_w,V)\rightarrow\rH^1(F_w,W),
\]
and define the Selmer groups ${\rm Sel}_?(F,T)$ and ${\rm Sel}_?(F,W)$ by the same recipe as in Definition~\ref{def:BDP-Sel}.

Let ${\rm Sel}_{p^k}(E/F)\subset\rH^1(G_{F,\Sigma},E[p^k])$ be the Selmer group fitting into the $p^k$-descent exact sequence
\[
0\rightarrow E(F)/p^kE(F)\rightarrow{\rm Sel}_{p^k}(E/F)\rightarrow\Sha(E/F)[p^k]\rightarrow 0,
\]
where $\Sha(E/F)[p^k]$ denotes the $p^k$-torsion in the Tate--Shafarevich group of $E/F$. Set 
\[
{\rm Sel}_{p^\infty}(E/F)=\varinjlim_k{\rm Sel}_{p^k}(E/F),\qquad\qquad
\check{S}_p(E/F)=\varprojlim_k{\rm Sel}_{p^k}(E/F),
\]
where the limits are with respect to the inclusion $E[p^k]\rightarrow E[p^{k+1}]$ and the multiplication-by-$p$ map $E[p^{k+1}]\rightarrow E[p^k]$, respectively.
Finally, set $K_p:=K\otimes\Q_p\simeq K_\pp\oplus K_{\ppbar}$.

\begin{lem}\label{lem:rank-1}
Assume {\rm \eqref{eq:h0-intro}}-{\rm \eqref{eq:h1-intro}}.
Then for every $\qq\in\{\pp,\ppbar\}$
\[
{\rm Sel}_\qq(K,T)=\ker({\rm res}_{\qq/{\rm tor}}),
\] 
where ${\rm res}_{\qq/{\rm tor}}$ is the composition
\[
\check{S}_p(E/K)\xrightarrow{{\rm res}_{\qq}}E(K_\qq)\otimes\Z_p\rightarrow E(K_\qq)_{/{\rm tor}}\otimes\Z_p.
\] 
In particular, 
\[
{\rm rank}_{\Z_p}{\rm Sel}_\qq(K,T)={\rm rank}_{\Z_p}\check{S}_p(E/K)-1.
\]
\end{lem}

\begin{proof}
It follows from global duality and Tate's global Euler characteristic formula that the image of the restriction map
\[
{\rm Sel}_{\rm rel}(K,T)\rightarrow\rH^1(K_p,T):=\rH^1(K_\pp,T)\oplus\rH^1(K_{\ppbar},T)
\]
has $\Z_p$-rank $2$ (see \cite[Lem.\,2.3.1]{skinner}).
By global duality we also have the exact sequence
\begin{equation}\label{eq:PT-ranks}
\begin{aligned}
0\rightarrow{\rm Sel}_{\rm str}(K,T)\rightarrow\check{S}_p(E/K)\xrightarrow{{\rm res}_{p/{\rm tor}}}&(E(K_\pp)_{/{\rm tor}}\otimes\Z_p)\oplus(E(K_{\ppbar})_{/{\rm tor}}\otimes\Z_p)\\
&\rightarrow{\rm Sel}_{\rm rel}(K,W)^\vee\rightarrow{\rm Sel}_{p^\infty}(E/K)^\vee\rightarrow 0.
\end{aligned}
\end{equation}
Note that by assumption, the $\Z_p$-rank of the image of the map ${\rm res}_{p/{\rm tor}}={\rm res}_{\pp/{\rm tor}}\oplus{\rm res}_{\ppbar/{\rm tor}}$ is either $1$ or $2$.
In the former case, it follows from \cite[Lem.\,2.3.2]{skinner} that
\begin{equation}
\label{eq:BDP=str}
{\rm Sel}_\pp(K,T)={\rm Sel}_{\rm str}(K,T)={\rm Sel}_{\ppbar}(K,T)
\end{equation}
and this yields the conclusion of the lemma. 
If the image of ${\rm res}_{p/{\rm tor}}$ has $\Z_p$-rank $2$, 
we see from \eqref{eq:PT-ranks} that ${\rm Sel}_{p^\infty}(E/K)$ is contained in ${\rm Sel}_{\rm rel}(K,W)$ with finite index, and the conclusion again follows.
\end{proof}

\begin{rem}
A discussion of the case where ${\rm im}({\rm res}_{p/{\rm tor}})$ has $\Z_p$-rank $2$ is missing in the proof of \cite[Lem.\,2.2]{AC}.
Hence the statement of that lemma needs to be corrected as in Lemma~\ref{lem:rank-1}.
\end{rem}

For $\qq\in\{\pp,\ppbar\}$ set
\[
{\rm Sel}_\qq(K_\infty,W)=\varinjlim_n{\rm Sel}_\qq(K_n,W).
\]
As is well-known, 
${\rm Sel}_{\mathfrak{q}}(K_\infty,W)$ is a cofinitely generated $\Lambda$-module, i.e., its Pontryagin dual ${\rm Sel}_{\mathfrak{q}}(K_\infty,W)^\vee$ is finitely generated over $\Lambda$.

\begin{conj}[Iwasawa--Greenberg Main Conjecture]\label{conj:BDP-IMC}
${\rm Sel}_{\ppbar}(K_\infty,W)$ is $\Lambda$-cotorsion and 
\[
{\rm char}_\Lambda({\rm Sel}_{\ppbar}(K_\infty,W)^\vee)\Lambda_{\widehat{\cO}}=(L_\pp^{\rm BDP}).
\]
\end{conj}

Here, as in the Introduction, $L_\pp^{\rm BDP}=(\mathscr{L}_\pp^{\rm BDP})^2$ denotes the square of the $p$-adic $L$-function 
\[
\mathscr{L}_\pp^{\rm BDP}\in\Lambda_{\widehat{\cO}}
\] 
introduced in \cite{bdp1} and further studied in \cite{braIMRN} (where it was shown to be an element in the Iwasawa algebra) and \cite{cas-hsieh1} (where it was also shown to be non-zero).

The first claim in Conjecture~\ref{conj:BDP-IMC} is now known under a mild hypothesis. 

\begin{prop}\label{prop:Lambda-tors}
Assume \eqref{eq:h0-intro}. Then ${\rm Sel}_\qq(K_\infty,W)$ is $\Lambda$-cotorsion.
\end{prop}

\begin{proof}
In the $p$-ordinary case, this follows from \cite[Thm.~4.2.2]{CGLS} and \cite[Thm.~5.53]{CGS}. Note that the proof of these results is based on Kolyvagin's methods applied to the anticyclotomic Euler system of Heegner points on $E/K$ 
and a $\Lambda$-adic extension of the formula \eqref{eq:bdp} of Bertolini--Darmon--Prasanna obtained in \cite[Thm.~5.1]{cas-hsieh1}.
For supersingular primes $p$ (and assuming $a_p=0$ when $p=3$), the result similarly follows from 
an adaptation of the argument in \cite[Thm.~3.4.1]{CGLS} applied to (signed) Heegner point Kolyvagin system constructed in \cite[Thm.~A.4]{CW-PR} (which is non-trivial by \cite[Cor.~6.4]{CW-PR} and in a  similar relation to $L_\qq^{\rm BDP}$ as in the $p$-ordinary case by virtue of \cite[Thm.~6.2]{CW-PR}). Finally, when $p=3$ is supersingular for $E$ but $a_3$ is not necessarily zero, the result follows from the analogues of the aforementioned results from \cite{CW-PR} developed in \cite[\S{4}]{CCSS}.
\end{proof}


We conclude this section by recalling the following basic result 
which will be useful in some of our arguments.

\begin{prop}\label{prop:greenberg}
Assume \eqref{eq:h0-intro}. 
Then 
${\rm Sel}_\qq(K_\infty,W)$ has no proper finite index $\Lambda$-submodules.
\end{prop}

\begin{proof}
As shown in \cite[Prop.~3.12]{hatley-lei-MRL}, this can be deduced from Greenberg's general results \cite{greenberg-str-Sel}.
\end{proof}

\section{\texorpdfstring{Howard's derived $p$-adic heights and the derived regulator}{}}\label{sec:der-ht}

We keep the setting from Section~\ref{sec:Sel}, and assume in addition that \eqref{eq:h0-intro} holds. In this section, we recall Howard's general construction of derived $p$-adic heights \cite{howard-derived}, and extend some of his results in \emph{op.\,cit.} to our setting to obtain derived $p$-adic heights on ${\rm Sel}_\qq(K,T)$. Then we define the regulator ${\rm Reg}_{\pp,{\rm der},\gamma}$ appearing in the formulation of our conjectures.

%
%


\subsection{\texorpdfstring{$\Lambda_k$}{}-adic Selmer groups} \label{subsec:lambda-sel}

Fix a topological generator $\gamma\in\Gamma$.
Adopting the notations in \cite{howard-derived}, for any $k$ we set
\[
\cO_k=\Z_p/p^k\Z_p,\qquad\Lambda_k=\cO_k\dBr{\Gamma}.
\]
Let $\mathcal{K}_k$ be the localisation of $\Lambda_k$ at 
all elements of the form $g_n=\frac{\gamma^{p^n}-1}{\gamma-1}$ for some $n$, 
and define $\mathcal{P}_k$ by the exactness of the sequence
\begin{equation}\label{eq:K-P}
0\rightarrow\Lambda_k\rightarrow\mathcal{K}_k\rightarrow\mathcal{P}_k\rightarrow 0.
\end{equation}
Let $S^{(k)}=E[p^k]$, and put
\[
S_{\rm Iw}^{(k)}:=\varprojlim_n{\rm Ind}_{K_n/K}S^{(k)},\qquad
S_\infty^{(k)}:=\varinjlim_n{\rm Ind}_{K_n/K}S^{(k)},
\]
where the limits are with respect to the natural corestriction and restriction maps, respectively.

By Shapiro's lemma, we have a canonical $\Lambda_k[G_{K,\Sigma}]$-module isomorphism $S_{\rm Iw}^{(k)}\simeq S^{(k)}\otimes_{\cO_k}\Lambda_k$, where the $G_{K,\Sigma}$-action on $\Lambda_k$ is given by the inverse of the tautological character $G_{K,\Sigma}\twoheadrightarrow\Gamma\hookrightarrow\Lambda_k^\times$ (see \cite[Lem.~1.4]{howard-derived}).
There is also an isomorphism $S^{(k)}\otimes_{\cO_k}\cP_k\simeq S_\infty^{(k)}$ (\cite[Lem.~1.5]{howard-derived}) depending on the choice of $\gamma$.  
Tensoring \eqref{eq:K-P} over $\cO_k$ 
with $S^{(k)}$, we then obtain
\begin{equation}\label{eq:S:K-P}
0\rightarrow S_{\rm Iw}^{(k)}\rightarrow S_{\mathcal{K}}^{(k)}\rightarrow S_{\infty}^{(k)}\rightarrow 0,
\end{equation}
dropping the subscripts in $\cK_k$ for the ease of notation.

For $\qq\in\{\pp,\ppbar\}$, let $\mathcal{F}_{\qq}$ be the Selmer structure on $S_\cK^{(k)}$ given by
\begin{equation}
\label{eq:local-F}
\rH^1_{\mathcal{F}_{\qq}}(K_w,S_\cK^{(k)})=
\begin{cases}
\rH^1_{\rm unr}(K_w,S_{\cK}^{(k)})&\textrm{if $w\nmid p\infty$,}\\
\rH^1(K_{\overline{\qq}},S_{\cK}^{(k)})&\textrm{if $w=\overline{\qq}$,}\\
0&\textrm{if $w={\qq}$}.
\end{cases}
\end{equation}
Let $\mathfrak{Sel}_\qq(K,S_\cK^{(k)})$ be the associated Selmer group: 
\[
\mathfrak{Sel}_\qq(K,S_\cK^{(k)}):={\rm ker}\biggl\{\rH^1(G_{K,\Sigma},S_\cK^{(k)})\rightarrow\prod_{w\in\Sigma_f}\frac{\rH^1(K_w,S_\cK^{(k)})}{\rH^1_{\mathcal{F}_\qq}(K_w,S_\cK^{(k)})}\biggr\}.
\]

The short exact sequence \eqref{eq:S:K-P} induces the natural exact sequence
\[
\rH^1(K_w,S_{\rm Iw}^{(k)})\rightarrow \rH^1(K_w,S_{\mathcal{K}}^{(k)})\rightarrow \rH^1(K_w,S_\infty^{(k)}).
\]
Taking the image (resp. inverse image) of $\rH^1_{\mathcal{F}_\qq}(K_w,S_{\cK}^{(k)})$, we obtain the local condition  $\rH^1_{\mathcal{F}_\qq}(K_w,S_{\infty}^{(k)})$ (resp. $\rH^1_{\mathcal{F}_\qq}(K_w,S_{\rm Iw}^{(k)})$).  
Following \cite{MR-KS}, we refer to these local conditions as being \emph{propagated} from $\rH^1_{\mathcal{F}_\qq}(K_w,S_{\cK}^{(k)})$ (or $S_{\cK}^{(k)}$) via (\ref{eq:S:K-P}). The Selmer groups they define will be denoted $\mathfrak{Sel}_\qq(K,S_\infty^{(k)})$ and $\mathfrak{Sel}_{\qq}(K,S_{\rm Iw}^{(k)})$, respectively.  (Similar definitions of local conditions ``by propagation'' will be made below.)



\begin{defn}\label{def:Sel-k}
For every $\qq\in\{\pp,\ppbar\}$, denote by $\rH^1_{\mathcal{F}_\qq}(K_w,S^{(k)})$ the local conditions obtained from $\rH^1_{\mathcal{F}_\qq}(K_w,S_\infty^{(k)})$ by propagation via $S^{(k)}\rightarrow S^{(k)}_\infty$, and let 
\[
\mathfrak{Sel}_\qq(K,S^{(k)}):=\ker\Biggl\{\rH^1(G_{K,\Sigma},S^{(k)})\rightarrow\prod_{w\in\Sigma_f}\frac{\rH^1(K_w,S^{(k)})}{\rH^1_{\mathcal{F}_\qq}(K_w,S^{(k)})}\Biggr\}
\]
be the resulting Selmer group.
\end{defn}

Note that by condition \eqref{eq:h0-intro}, the natural surjective map $\rH^1(G_{K,\Sigma},S^{(k)})\twoheadrightarrow\rH^1(G_{K,\Sigma},S_\infty^{(k)})[J]$ 
is an isomorphism; since the local conditions on $\mathfrak{Sel}_\qq(K,S^{(k)})$ are propagated from $S_\infty^{(k)}$, this restricts to an isomorphism
\begin{equation}\label{eq:sp-iso}
\mathfrak{Sel}_\qq(K,S^{(k)})\simeq\mathfrak{Sel}_\qq(K,S_\infty^{(k)})[J].
\end{equation}
%

\subsection{Howard's derived \texorpdfstring{$p$}{}-adic heights}\label{subsec:hp-der}


For every $i\geq 1$, let $\mathfrak{S}_{\qq,k}^{(i)}\subset\mathfrak{Sel}_\qq(K,S^{(k)})$ be the submodule mapping to $J^{i-1}\mathfrak{Sel}_\qq(K,S_\infty^{(k)})[J]$ under the isomorphism \eqref{eq:sp-iso}.
We have a filtration
\[
\mathfrak{Sel}_\qq(K,S^{(k)})=\mathfrak{S}_{\qq,k}^{(1)}\supset\mathfrak{S}_{\qq,k}^{(2)}\supset\cdots\supset\mathfrak{S}_{\qq,k}^{(i)}\supset\cdots.
\]


\begin{thm}[Howard]\label{thm:howard-ht-pk}
For $i\geq 1$, there is a sequence of canonical 
symmetric \emph{$i$-th ``derived'' height pairings}
\[
h_{\qq,k}^{(i)}:\mathfrak{S}_{\qq,k}^{(i)}\times\mathfrak{S}_{\qqbar,k}^{(i)}\rightarrow\cO_k
\]
such that the kernel on the left (resp. right) is $\mathfrak{S}_{\qq,k}^{(i+1)}$ (resp. $\mathfrak{S}_{\qqbar,k}^{(i+1)}$).
\end{thm}

\begin{proof}
By definition \eqref{eq:local-F}, 
the local conditions cutting out  $\mathfrak{Sel}_\qq(K,S_\cK^{(k)})$ and $\mathfrak{Sel}_{\overline{\qq}}(K,S_\cK^{(k)})$, are everywhere exact orthogonal complements under the pairing 
\[
\rH^1(K_w,S_\cK^{(k)})\times \rH^1(K_w,S_\cK^{(k)})\rightarrow\rH^2(K_w,\cK(1))\simeq\cK
\]
induced by the Weil pairing $S^{(k)}\times S^{(k)}\rightarrow\cO_k(1)$.
Therefore, by \cite[Thm.~1.11]{howard-derived} there is a canonical symmetric 
\emph{height pairing}
\begin{equation}\label{eq:hk}
\tilde{h}_{\qq,k}:\mathfrak{Sel}_\qq(K,S_\infty^{(k)})\times\mathfrak{Sel}_{\overline{\qq}}(K,S_\infty^{(k)})\rightarrow\cO_k.\nonumber
\end{equation}
(As written here, the pairing $\tilde{h}_{\qq,k}$ depends on the choice of a topological generator $\gamma\in\Gamma$, but 
the $J/J^2$-valued pairing $(\gamma-1)\cdot\tilde{h}_{\qq,k}$ 
is independent of $\gamma$, as one checks immediately from \cite[Lem.~1.10(b)]{howard-derived}.) 
%
Note that $J^{i-1}\mathfrak{Sel}_\qq(K,S_\infty^{(k)})[J]$ is the image of the injection
\begin{equation}
\label{eq:phi-gamma}
\phi_{i,\gamma}:\frac{\mathfrak{Sel}_\qq(K,S_\infty^{(k)})[J^i]}{\mathfrak{Sel}_\qq(K,S_\infty^{(k)})[J^{i-1}]}\hookrightarrow\mathfrak{Sel}_\qq(K,S_\infty^{(k)})[J]
\end{equation}
given by multiplication by $(\gamma-1)^{i-1}$.
Thus we may define
\begin{equation}\label{eq:der-ht-k}
\tilde{h}_{\qq,k}^{(i)}:J^{i-1}\mathfrak{Sel}_\qq(K,S_\infty^{(k)})[J]\times J^{i-1}\mathfrak{Sel}_{\qqbar}(K,S_\infty^{(k)})[J]\rightarrow\cO_k\nonumber
\end{equation}
by $\tilde{h}_{\qq,k}^{(i)}(s,t):=\tilde{h}_{\qq,k}(\phi_{i,\gamma}^{-1}(s),t)$.
In particular, $\tilde{h}_{\qq,k}^{(1)}$ is the restriction of $\tilde{h}_{\qq,k}$ to $\mathfrak{Sel}_\qq(K,S_\infty^{(k)})[J]\times\mathfrak{Sel}_{\overline{\qq}}(K,S_\infty^{(k)})[J]$.
By \cite[Lem.~2.3]{howard-derived}, the left kernel (resp. right kernel) of $\tilde{h}_{\qq,k}^{(i)}$ is exactly $J^{i}\mathfrak{Sel}_\qq(K,S_\infty^{(k)})[J]$ (resp. $J^{i}\mathfrak{Sel}_{\qqbar}(K,S_\infty^{(k)})[J]$).
Since \eqref{eq:sp-iso} restricts to an isomorphism
\[
\mathfrak{S}_{\qq,k}^{(i)}\simeq J^{i-1}\mathfrak{Sel}_\qq(K,S_\infty^{(k)})[J],
\] 
we can transfer $\tilde{h}_{\qq,k}^{(i)}$ to $\mathfrak{S}_{\qq,k}^{(i)}\times\mathfrak{S}_{\qqbar,k}^{(i)}$ via this isomorphism.
Thus, we get a sequence of pairings $h_{\qq,k}^{(i)}$ with the stated properties.
\end{proof}


\subsection{Control theorems}

We now compare the Selmer groups $\mathfrak{Sel}_\qq(K,S_\infty^{(k)})$ (and limits thereof) of the preceding section with the Selmer groups ${\rm Sel}_\qq(K,T)$ and ${\rm Sel}_\qq(K_\infty,W)$ (see Corollary~\ref{cor:control} and Corollary~\ref{cor:control-mazur}). This will allow us to deduce from Theorem~\ref{thm:howard-ht-pk} a construction of $p$-adic height pairings for ${\rm Sel}_\qq(K,T)$, and to relate  their degeneracies to the $\Lambda$-module structure of ${\rm Sel}_\qq(K_\infty,W)$. 

Let ${\rm Sel}_\qq(K,S_\infty^{(k)})$ be the Selmer group cut out by the local conditions 
\[
\rH^1_{\qq}(K_w,S_\infty^{(k)})=
\begin{cases}
\rH^1_{\rm unr}(K_w,S_{\infty}^{(k)})&\textrm{if $w\nmid p\infty$,}\\
\rH^1(K_{\overline\qq},S_\infty^{(k)})&\textrm{if $w=\overline{\qq}$,}\\
0&\textrm{if $w={\qq}$}.
\end{cases}
\] 
Putting $S_\infty=\varinjlim_kS_\infty^{(k)}$, we also consider the $\Lambda=\Z_p\dBr{\Gamma}$-module
\[
{\rm Sel}_\qq(K,S_\infty):=\varinjlim_k{\rm Sel}_\qq(K,S_\infty^{(k)}),\quad
\]
where the limit is with respect to the maps induced by the inclusion $S^{(k)}\rightarrow S^{(k+1)}$. 

\begin{lem}\label{lem:shapiro}
For every $\qq\in\{\pp,\ppbar\}$ there is a canonical $\Lambda$-module isomorphism
\[
{\rm Sel}_\qq(K,S_\infty)\simeq{\rm Sel}_\qq(K_\infty,W).
\]
\end{lem}

\begin{proof}
By Shapiro's lemma, we have $\rH^1(G_{K,\Sigma},S_\infty)\simeq\rH^1(G_{K_\infty,\Sigma},W)$. We check that under the above isomorphism, the Selmer groups in the statement are cut out by the same local conditions. For the  primes $w$ above $p$, this is clear. For the primes $w\in\Sigma_f\setminus\{\pp,\ppbar\}$, we need to check that $\varinjlim_k\rH^1_{\rm unr}(K_w,S_\infty^{(k)})=0$, but this follows from \cite[Lem.~1.7]{howard-derived}, since by our assumption on $\Sigma$ 
these primes are split in $K$, so they are finitely decomposed in $K_\infty/K$.  \end{proof}

\begin{rem}In light of Lemma~\ref{lem:shapiro}, 
we shall henceforth write ${\rm Sel}_\qq(K,S_\infty)$ and ${\rm Sel}_\qq(K_\infty,W)$ interchangeably.
\end{rem}

Directly from the definitions we have the inclusion $\mathfrak{Sel}_\qq(K,S_\infty^{(k)})\subset{\rm Sel}_\qq(K,S_\infty^{(k)})$.
On the other hand, the natural surjection 
\begin{equation}\label{eq:H1-modpk}
\rH^1(G_{K,\Sigma},S_\infty^{(k)})\rightarrow\rH^1(G_{K,\Sigma},S_\infty)[p^k]
\end{equation} 
induces a map $\alpha_\qq:{\rm Sel}_\qq(K,S_\infty^{(k)})\rightarrow{\rm Sel}_\qq(K,S_\infty)[p^k]$.

\begin{prop}\label{prop:control}
For every $\qq\in\{\pp,\ppbar\}$, the composition
\[
\mathfrak{Sel}_\qq(K,S_\infty^{(k)})\xrightarrow{\subset}{\rm Sel}_\qq(K,S_\infty^{(k)})\xrightarrow{\alpha_\qq}{\rm Sel}_\qq(K,S_\infty)[p^k]
\]
is injective with finite cokernel of bounded order as $k\to\infty$.
\end{prop}

\begin{proof}
By \eqref{eq:h0-intro}, the map \eqref{eq:H1-modpk} is an injection, and therefore so is $\alpha_\qq$.
To bound the cokernel of the map in the statement, we bound the cokernel of each of the two maps in the composition.

From the definitions, we see that the quotient ${\rm Sel}_\qq(K,S_\infty^{(k)})/\mathfrak{Sel}_\qq(K,S_\infty^{(k)})$ injects into
\begin{equation}\label{eq:def-Fq}
\frac{\rH^1(K_{\overline\qq},S_\infty^{(k)})}{{\rm im}\bigl\{\rH^1(K_{\overline{\qq}},S_\cK^{(k)})\rightarrow\rH^1(K_{\overline{\qq}},S_\infty^{(k)})\bigr\}}\simeq{\rm ker}\bigl\{\rH^2(K_{\overline{\qq}},S_{\rm Iw}^{(k)})\rightarrow\rH^2(K_{\overline{\qq}},S_\cK^{(k)})\bigr\}.
\end{equation}
By local duality, this is bounded by the size of 
$\bigoplus_{\eta\mid\overline{\qq}}E(K_{\infty,\eta})[p^k]$.
Since $\bigoplus_{\eta\mid\overline{\qq}}E(K_{\infty,\eta})[p^\infty]$ is finite by \cite[Lem.~2.7]{KO}, the above quotient has the desired bound.


On the other hand, for the primes $w\in\Sigma_f\setminus\{\pp,\ppbar\}$,  \cite[Lem.~1.7]{howard-derived} gives
\[
\rH^1_{\rm unr}(K_w,S_\infty^{(k)})=\rH^1_{\rm unr}(K_w,S_\infty)=0.
\] 
Therefore, by the snake lemma we see that the cokernel of $\alpha_\qq$ is bounded by the kernel of the restriction map
\begin{equation}
\label{eq:res-a2}
(r_w)_w:\bigoplus_{w\in\Sigma_f\setminus\{\overline{\qq}\}}\rH^1(K_w,S_\infty^{(k)})\rightarrow\bigoplus_{w\in\Sigma_f\setminus\{\overline{\qq}\}}\rH^1(K_w,S_\infty)[p^k].
\end{equation}
From the cohomology long exact sequence associated with multiplication by $p^k$ we see that ${\rm ker}(r_w)=\rH^0(K_w,S_\infty)/p^k\rH^0(K_w,S_\infty)$, and this is bounded by 
\begin{equation}\label{eq:def-B}
\prod_{w\in\Sigma_f\setminus\{\overline{\qq}\}}\#((B_w)_{/{\rm div}}),\quad\textrm{where $B_w:=\bigoplus_{\eta\mid w}E(K_{\infty,\eta})[p^\infty]$,}
\end{equation}
which is clearly finite and independent of $k$. Thus $\#{\rm coker}(\alpha_\qq)$ has bounded order as $k\to\infty$, whence the result.
\end{proof}

\begin{cor}\label{cor:control}
With notation as above, the following equality of $\Lambda$-modules holds
\[
\varinjlim_k\mathfrak{Sel}_\qq(K,S_\infty^{(k)})={\rm Sel}_\qq(K,S_\infty).
\]
\end{cor}

\begin{proof}
By Proposition~\ref{prop:control}, $\varinjlim_k\mathfrak{Sel}_\qq(K,S_\infty^{(k)})$ is contained in ${\rm Sel}_\qq(K,S_\infty)$ with finite index, so the result follows from Proposition~\ref{prop:greenberg}.
\end{proof}

The next result is a variant of Mazur's control theorem for our $\qq$-strict Selmer groups.

\begin{prop}\label{prop:control-mazur}
The map $W\rightarrow S_\infty$ induces an injection
\[
\beta_\qq:{\rm Sel}_\qq(K,W)\rightarrow{\rm Sel}_\qq(K,S_\infty)[J]
\]
with finite cokernel.
\end{prop}

\begin{proof}
The map $\beta_\qq$ is the restriction of the natural map $\rH^1(G_{K,\Sigma},W)\rightarrow\rH^1(G_{K,\Sigma},S_\infty)$ induced by $W\rightarrow S_\infty$.
This natural map is part of the cohomology long exact sequence  induced by 
\[
0\rightarrow W\rightarrow S_\infty\xrightarrow{\gamma-1}S_\infty\rightarrow 0,
\]
and by \eqref{eq:h0-intro}, it induces an isomorphism $\rH^1(G_{K,\Sigma},W)\simeq\rH^1(G_{K,\Sigma},S_\infty)[J]$.
The injectivity of $\beta_\qq$ follows from this.
On the other hand, from the definitions and a direct application of the snake lemma, we see that the cokernel of $\beta_\qq$ is bounded by the kernel of the restriction map
\begin{equation}
\label{eq:res-qbar-w}
(r_{\overline\qq},(r_w)_w):\frac{\rH^1(K_{\overline\qq},W)}{\rH^1(K_{\overline\qq},W)_{\rm div}}\times\prod_{w\in\Sigma_f\setminus\{\qqbar\}}\rH^1(K_w,W)\rightarrow\{0\}\times\prod_{w\in\Sigma_f\setminus\{\qqbar\}}\prod_{\eta\mid w}\rH^1(K_{\infty,\eta},W).
\end{equation}

For $w=\overline\qq$, we compute
\begin{equation}
\label{eq:q-tor}
\begin{aligned}
{\rm ker}(r_{\overline\qq})
&=\rH^1(K_{\overline\qq},T)_{\rm tors} \quad \textrm{by local Tate duality} \\
&={\rm ker}\{\rH^1(K_{\overline\qq},T)\rightarrow\rH^1(K_{\overline\qq},V)\}\\
&={\rm coker}\{\rH^0(K_{\overline\qq},V)\rightarrow\rH^0(K_{\overline\qq},W)\},
\end{aligned}
\end{equation}
where the last equality follows from the cohomology long exact sequence associated to $T\hookrightarrow V\twoheadrightarrow W$.
Since $\rH^0(K_{\overline\qq},V)=0$, this shows that 
\begin{equation}\label{eq:q-tor-bis}
\#{\rm ker}(r_{\overline\qq})=\#\rH^0(K_{\overline\qq},W)=\#E(\Q_p)[p^\infty].
\end{equation}
For $w={\qq}$, we have
\[
{\rm ker}(r_{{\qq}})=B_{{\qq}}/(\gamma-1)B_{{\qq}},
\]
where $B_{\qq}$ is as in \eqref{eq:def-B}, and therefore finite by \cite[Lem.~2.7]{KO}.
Finally, for $w\in\Sigma_f\setminus\{\pp,\ppbar\}$ we have ${\rm ker}(r_w)=B_w/(\gamma-1)B_w$.
From the exact sequence
\[
0\rightarrow E(K_w)[p^\infty]\rightarrow B_w\xrightarrow{\gamma-1}B_w\rightarrow B_w/(\gamma-1)B_w\rightarrow 0
\]
and the finiteness of $E(K_w)[p^\infty]$, we see that 
\[
(B_w)_{\rm div}\subset(\gamma-1)B_w,
\]
and so $\#{\rm ker}(r_w)$ is bounded by $[B_w:(B_w)_{\rm div}]$.
Since all primes $w\in\Sigma_f$ are finitely decomposed in $K_\infty/K$ by our assumption on $\Sigma$, this concludes the proof.
\end{proof}

\begin{rem}\label{rem:jsw}
When $\#{\rm Sel}_\qq(K,W)<\infty$, adapting the arguments in \cite[\S{4}]{greenberg-cetraro}, one can determine the exact size of the cokernel of the restriction map $\beta_\qq$ in Proposition~\ref{prop:control-mazur}, resulting in the formula 
\[
\#{\rm coker}(\beta_\qq)=(\#E(\Q_p)[p^\infty])^2\cdot\prod_{w\mid N}c_w^{(p)},
\] 
where $c_w^{(p)}$ is the $p$-part of the Tamagawa number of $E/K_w$ (see \cite[Thm.~3.3.1]{JSW}).
However, Greenberg's arguments rely crucially on the surjectivity
of the global-to-local map
\[
\rH^1(G_{K,\Sigma},W)\rightarrow\prod_{w\in\Sigma_f}\frac{\rH^1(K_w,W)}{\rH^1_\qq(K_w,W)},
\]
which fails when ${\rm Sel}_\qq(K,W)$ is infinite. In our approach, when ${\rm Sel}_\qq(K,W)$ is not necessarily finite, a result playing the role of  an exact control on $\#{\rm coker}(\beta_\qq)$ will be obtained in $\S\ref{subsec:step3}$ (see Corollary~\ref{cor:control-2}).
\end{rem}

\begin{cor}\label{cor:control-mazur}
The generalised Selmer group $\varprojlim_k\mathfrak{Sel}_\qq(K,S_\infty^{(k)})[J]$ is contained in ${\rm Sel}_\qq(K,T)$ with finite index.
\end{cor}

\begin{proof}
By \eqref{eq:h0-intro}, the natural surjection $\rH^1(G_{K,\Sigma},S^{(k)})\rightarrow\rH^1(G_{K,\Sigma},W)[p^k]$ is an isomorphism. 
Since the local conditions defining ${\rm Sel}_\qq(K,T)$ and ${\rm Sel}_\qq(K,W)$ are propagated from $\rH^1_\qq(K_w,V)$, we have
\begin{equation}\label{eq:iso-T}
{\rm Sel}_\qq(K,T)\simeq\varprojlim_k{\rm Sel}_\qq(K,W)[p^k].
\end{equation}

On the other hand, it follows from the proof of Proposition~\ref{prop:control-mazur} that for every $k$ there is a natural injection 
\begin{equation}
{\rm Sel}_\qq(K,W)[p^k]\rightarrow{\rm Sel}_\qq(K,S_\infty)[J+p^k\Lambda]\nonumber
\end{equation}
with cokernel bounded by the size of 
\[
\biggl(E(\Q_p)[p^\infty]\times\prod_{w\in\Sigma_f\setminus\{\qqbar\}}B_w/(\gamma-1)B_w\biggr)[p^k].
\]
Since this is finite (even before taking $p^k$-torsion) and the transition maps are given by multiplication by $p$, its inverse limit with respect to $k$ vanishes. Therefore,
\begin{equation}\label{eq:iso-mazur-control}
\varprojlim_k{\rm Sel}_\qq(K,W)[p^k]\simeq\varprojlim_k{\rm Sel}_\qq(K,S_\infty)[J+p^k\Lambda].
\end{equation}

Since Proposition~\ref{prop:control} implies that $\varprojlim_k\mathfrak{Sel}_\qq(K,S_\infty^{(k)})[J]$ is contained in $\varprojlim_k{\rm Sel}_\qq(K,S_\infty)[J+p^k\Lambda]$ with finite index, the result follows from \eqref{eq:iso-T} and \eqref{eq:iso-mazur-control}.
\end{proof}

\begin{defn}
\label{def:proj}
For every $i\geq 1$, put
\[
\underleftarrow{\mathfrak{S}}_\qq^{(i)}=\varprojlim_k\mathfrak{S}_{\qq,k}^{(i)},
\] 
where the limit is with respect to the multiplication-by-$p$ maps $S^{(k+1)}\rightarrow S^{(k)}$.
\end{defn}

Thus we obtain a filtration
\begin{equation}
\label{eq:fil-T}
\varprojlim_k\mathfrak{Sel}_{\qq}(K,S^{(k)})=
\underleftarrow{\mathfrak{S}}_\qq^{(1)}\supset\underleftarrow{\mathfrak{S}}_\qq^{(2)}\supset\cdots\supset\underleftarrow{\mathfrak{S}}_\qq^{(i)}\supset\cdots\supset\underleftarrow{\mathfrak{S}}_\qq^{(\infty)},
\end{equation}
where $\underleftarrow{\mathfrak{S}}_\qq^{(\infty)}:=\cap_{i\geq 1}\underleftarrow{\mathfrak{S}}_\qq^{(i)}$. The pairings $h_{\qq,k}^{(i)}$ of Theorem~\ref{thm:howard-ht-pk} are compatible as $k$ varies, and in the limit they give rise to a sequence of ``derived'' $p$-adic height pairings
\begin{equation}
\label{eq:ht-Zp}
h_{\qq}^{(i)}:\underleftarrow{\mathfrak{S}}_\qq^{(i)}\times\underleftarrow{\mathfrak{S}}_{\qqbar}^{(i)}\rightarrow\Z_p
\end{equation}
such that the kernel on the left (resp. right) is $\underleftarrow{\mathfrak{S}}_\qq^{(i+1)}$ (resp. $\underleftarrow{\mathfrak{S}}_{\qqbar}^{(i+1)}$).

%

\begin{cor}
\label{cor:der-ht}
Let $\qq\in\{\pp,\ppbar\}$.
Using $\sim$ to denote $\Lambda$-module pseudo-isomorphism, write
\[
{\rm Sel}_\qq(K,S_\infty)^\vee\,\sim\,\Lambda^{e_\infty}\oplus(\Lambda/J)^{e_1}\oplus(\Lambda/J^2)^{e_2}\oplus\cdots\oplus(\Lambda/J^{i})^{e_{i}}\oplus\cdots\oplus M
\]
with $M$ a torsion $\Lambda$-module with characteristic ideal prime to $J$.
Then 
\begin{itemize}
\item[(i)] $e_i={\rm rank}_{\Z_p}(\underleftarrow{\mathfrak{S}}_\qq^{(i)}/\underleftarrow{\mathfrak{S}}_\qq^{(i+1)})$.
\item[(ii)] $e_\infty={\rm rank}_{\Z_p}(\underleftarrow{\mathfrak{S}}_{\qq}^{(\infty)})={\rm rank}_{\Z_p}({\rm Sel}_\qq(K,T)^u)$, where 
\[
{\rm Sel}_\qq(K,T)^u:=\bigcap_{n}{\rm cor}_{K_n/K}({\rm Sel}_\qq(K_n,T))
\]
is the space of universal norms in ${\rm Sel}_\qq(K,T)$.
\end{itemize}
\end{cor}
\begin{rem}
Proposition~\ref{prop:Lambda-tors} says that ${\rm Sel}_\qq(K,S_\infty)^\vee$ is $\Lambda$-torsion, so $e_\infty=0$.
\end{rem}

\begin{proof}
The argument leading to the proof of \cite[Cor.~4.3]{howard-derived} (for the usual Selmer group) applies \emph{verbatim} to our setting, replacing the use of Proposition~3.5 and Lemma~4.1 in \emph{op.\,cit.} by Proposition~\ref{prop:control} and (the proof of) Proposition~\ref{prop:control-mazur} above, respectively.
\end{proof}

\subsection{The derived regulator}\label{subsec:der-reg}




%
Note that by Corollary~\ref{cor:control-mazur} and 
\eqref{eq:sp-iso}, the Selmer group 
\[
\varprojlim_k\mathfrak{Sel}_{\qq}(K,S^{(k)})= \underleftarrow{\mathfrak{S}}_\qq^{(1)}
\] 
is contained in ${\rm Sel}_\qq(K,T)$ with finite index.

\begin{defn}
For $i\geq 1$, define
\[
\underleftarrow{\overline{\mathfrak{S}}}_\qq^{(i)}:=\bigl(\underleftarrow{\mathfrak{S}}_\qq^{(i)}\otimes_{\Z_p}\Q_p\bigr)\cap{\rm Sel}_\qq(K,T)
\] 
to be the $p$-adic saturation of $\underleftarrow{\mathfrak{S}}_\qq^{(i)}$ in ${\rm Sel}_\qq(K,T)$.
In particular, $\underleftarrow{\overline{\mathfrak{S}}}_\qq^{(1)}={\rm Sel}_\qq(K,T)$.
\end{defn}
By linearity, the pairings \eqref{eq:ht-Zp} extend to $\Q_p$-valued pairings
\begin{equation}\label{eq:ht-Qp}
{h}_\qq^{(i)}:\underleftarrow{\overline{\mathfrak{S}}}_\qq^{(i)}\times\underleftarrow{\overline{\mathfrak{S}}}_{\qqbar}^{(i)}\rightarrow\Q_p,
\end{equation}
whose kernel on the left (resp. right) is $\underleftarrow{\overline{\mathfrak{S}}}_\qq^{(i+1)}$ (resp. $\underleftarrow{\overline{\mathfrak{S}}}_{\qqbar}^{(i+1)}$) by virtue of Theorem~\ref{thm:howard-ht-pk}.



By definition, for every $i\geq 1$ the quotients $\underleftarrow{\overline{\mathfrak{S}}}_\pp^{(i)}/\underleftarrow{\overline{\mathfrak{S}}}_{\pp}^{(i+1)}$ and $\underleftarrow{\overline{\mathfrak{S}}}_{\ppbar}^{(i)}/\underleftarrow{\overline{\mathfrak{S}}}_{\ppbar}^{(i+1)}$ are free $\Z_p$-modules, and since the action of complex conjugation defines an isomorphism ${\rm Sel}_\pp(K,S_\infty)\simeq{\rm Sel}_{\ppbar}(K,S_\infty)$, by Corollary~\ref{cor:der-ht}(i) the $\Z_p$-ranks of these quotients are equal.
Hence for every $\qq\in\{\pp,\ppbar\}$ we have
\begin{equation}\label{eq:ei}
\underleftarrow{\overline{\mathfrak{S}}}_\qq^{(i)}/\underleftarrow{\overline{\mathfrak{S}}}_{\qq}^{(i+1)}\simeq\Z_p^{e_i}
\end{equation}
for some integers $e_i\geq 0$ (the same for both $\pp$ and $\ppbar$).
By Corollary~\ref{cor:der-ht}(ii) and Proposition~\ref{prop:Lambda-tors} we know that $e_i=0$ for $i\gg 0$.


Note that the $i$-th derived $p$-adic heights $h_\qq^{(i)}$ depend on the choice of a topological generator $\gamma$, but the $(J^{i}/J^{i+1})\otimes_{\Z_p}\Q_p$-valued pairings $(\gamma-1)^{i}\cdot h_{\qq}^{(i)}$ are independent of this choice.
We record the dependence on $\gamma$ in the following definition.

\begin{defn}\label{def:der-Reg}
Let $(x_{1,1},\dots,x_{1,e_1};\cdots;x_{i,1},\dots,x_{i,e_i};\cdots)$ be a $\Z_p$-basis for ${\rm Sel}_\pp(K,T)$ adapted to the filtration
\begin{equation}\label{eq:fil-T-sat}
{\rm Sel}_\pp(K,T)= \underleftarrow{\overline{\mathfrak{S}}}_\pp^{(1)}\supset\underleftarrow{\overline{\mathfrak{S}}}_\pp^{(2)}\supset\cdots\supset\underleftarrow{\overline{\mathfrak{S}}}_\pp^{(i)}\supset\cdots,\nonumber
\end{equation}
so that for every $i\geq 1$ the elements $x_{i,1},\dots,x_{i,e_i}$ project to a $\Z_p$-basis for $\underleftarrow{\overline{\mathfrak{S}}}_\pp^{(i)}/\underleftarrow{\overline{\mathfrak{S}}}_\pp^{(i+1)}$.
Let $(y_{1,1},\dots,y_{1,e_1};\cdots;y_{i,1},\dots,y_{i,e_i};\cdots)$ be a $\Z_p$-basis for ${\rm Sel}_{\ppbar}(K,T)$ defined in the same manner.
Define the \emph{$i$-th partial regulator} 
by
\begin{equation}\label{eq:partialp}
R_{\pp,\gamma}^{(i)}
:=\det\bigl(h_\pp^{(i)}(x_{i,j},y_{i,j'})\bigr)_{1\leq j,j'\leq e_i},\nonumber
\end{equation}
and the \emph{derived regulator} 
by ${\rm Reg}_{\pp,{\rm der},\gamma}:=\prod_{i\geq 1}R_{\pp,\gamma}^{(i)}$.
\end{defn}

\begin{rem}\label{rem:J-valued}
By definition, the partial regulators $R^{(i)}_{\pp,\gamma}$ are non-zero, and they are well-defined up to a $p$-adic unit.
So, we have
\[
{\rm Reg}_{\pp,{\rm der},\gamma}\in\Q_p^\times/\Z_p^\times.
\]
Replacing $h_{\pp}^{(i)}$ by $(\gamma-1)^i\cdot h_{\pp}^{(i)}$ and writing $\sigma=\sum_{i\geq 1}ie_{i}$, the above definition gives a non-zero derived regulator
\[
{\rm Reg}_{\pp,{\rm der}}\in\bigl((J^\sigma/J^{\sigma+1})\otimes_{\Z_p}\Q_p\bigr)/\Z_p^\times
\]
which is independent of $\gamma$.
\end{rem}


\section{BSD conjecture for \texorpdfstring{$L_\pp^{\rm BDP}$}{}}\label{sec:BSDconj}

In this section, we formulate our $p$-adic analogue of the Birch--Swinnerton-Dyer conjecture for $L_\pp^{\rm BDP}$, extending the formulation given in \cite{AC} in the $p$-ordinary case. 

We keep the setting from Section~\ref{sec:Sel}, and assume in addition that \eqref{eq:h0-intro} and \eqref{eq:h1-intro} hold.



By 
Lemma~\ref{lem:rank-1}, we have 
\[
{\rm Sel}_\pp(K,T)={\rm ker}({\rm res}_{\pp/{\rm tor}})\simeq\Z_p^{r-1},
\]
where $r={\rm rank}_{\Z_p}\check{S}_p(E/K)$.
Let $(s_1,\dots,s_{r-1})$ be a $\Z_p$-basis for ${\rm Sel}_\pp(K,T)$, and extend it to a $\Z_p$-basis $(s_1,\dots,s_{r-1},s_\pp)$ for $\check{S}_p(E/K)$.
In particular, ${\rm res}_\pp(s_\pp)\in E(K_\pp)\otimes\Z_p$ is non-torsion.
Henceforth, we let
\[
{\rm log}_\pp:\check{S}_p(E/K)\rightarrow\Z_p
\]
be the composition of the map ${\rm res}_{\pp/{\rm tor}}$ with the formal group logarithm $E(K_\pp)_{/{\rm tor}}\otimes\Z_p\rightarrow\Z_p$ associated with a N\'{e}ron differential $\omega_E\in\Omega^1(E/\Z_{(p)})$.
Also, let $\Sha_{\rm BK}(K,W)={\rm Sel}_{p^\infty}(E/K)_{/{\rm div}}$ be the Bloch--Kato Tate--Shafarevich group, and for every prime $\ell\mid N$ write $c_\ell$ to denote the Tamagawa number of $E/\Q_\ell$. 

In the following, we shall interchangeably view $L_\pp^{\rm BDP}$ as an element in $\widehat\cO\dBr{T}$ via the identification $\Lambda_{\widehat{\cO}}\simeq\widehat\cO\dBr{T}$ defined by $T=\gamma-1$ for our fixed topological generator $\gamma\in\Gamma$. Thus $T$ corresponds to a generator of the augmentation ideal $J\subset\Lambda_{\widehat{\cO}}$. 



\begin{conj}[$p$-adic BSD conjecture for $L_\pp^{\rm BDP}$]\label{conj:BSD}
The following assertions hold:
\begin{itemize}
\item[(i)] {\rm (Leading Coefficient Formula)} Let 
$\varrho_{\rm an}:={\rm ord}_JL_\pp^{\rm BDP}$. 
Then, up to a $p$-adic unit, 
\begin{align*}
\frac{1}{\varrho_{\rm an}!}\frac{d^{\varrho_{\rm an}}}{dT^{\varrho_{\rm an}}}L_{\pp}^{\rm BDP}\Bigl\vert_{T=0}&=\biggl(\frac{1-a_p(E)+p}{p}\biggr)^2\cdot{\rm log}_\pp(s_\pp)^2\\
&\quad\times{\rm Reg}_{\pp,{\rm der},\gamma}\cdot\#\Sha_{\rm BK}(K,W)\cdot\prod_{\ell\mid N}c_\ell^2.
\end{align*}
\item[(ii)] {\rm (Order of Vanishing)}
Set $r^\pm$ to denote the $\Z_p$-corank of the $\pm$-eigenspace of $\check{S}_p(E/K)$ under the action of complex conjugation.
Then
\[
\varrho_{\rm an}=2(\max\{r^+,r^-\}-1).
\]
\end{itemize}
\end{conj}


\begin{rem}
We observe that Conjecture~\ref{conj:BSD} is a reformulation (depending on $\gamma$) of Conjecture~\ref{conj:BSD-intro} in the Introduction.
\end{rem}
 
\begin{rem}
As noted in the Introduction (see Remark~\ref{rem:cf-AC}),  Conjecture~\ref{conj:BSD} extends to all good primes $p>2$ and with an unconditional definition of ${\rm Reg}_{\pp,{\rm der}}$ the $p$-adic Birch--Swinnerton-Dyer conjecture formulated in \cite[Conj.~4.2]{AC}. That their $p$-adic heights (deduced from the construction in \cite{BD-derived-AJM}) agree with ours (deduced from the construction in \cite{howard-derived}) when both apply, follows from the height comparisons in \cite[\S{11}]{nekovar302} and \cite[\S{10}]{burns-mc}.
\end{rem}


By the works of Kolyvagin, Gross--Zagier, and Bertolini--Darmon--Prasanna, Conjecture~\ref{conj:BSD} enjoys the following compatibility with the  classical Birch--Swinnerton-Dyer for $L(E/K,s)$.

\begin{prop}\label{prop:rank-1}
Assume ${\rm ord}_{s=1}L(E/K,s)=1$. Then: 
\begin{itemize}
\item[(i)] Conjecture~\ref{conj:BSD}(i) is equivalent to the $p$-part of the Birch--Swinnerton-Dyler formula for $L'(E/K,1)$.
\item[(ii)] $\varrho_{\rm an}=0$, and Conjecture~\ref{conj:BSD}(ii) holds.
\end{itemize}
\end{prop}

\begin{proof}
By the Gross--Zagier formula \cite{GZ}, the Heegner point $z_K\in E(K)$ in formula \eqref{eq:bdp} is non-torsion. Therefore, ${\rm rank}_{\Z}E(K)=1$ and $\#\Sha(E/K)<\infty$ by Kolyvagin's work \cite{kol88}, and $\varrho_{\rm an}=0$ by formula \eqref{eq:bdp}. In particular, Conjecture~\ref{conj:BSD}(ii) holds. 

The above also shows that $\Sha_{\rm BK}(K,W)=\Sha(E/K)[p^\infty]$ and ${\rm rank}_{\Z_p}\check{S}_p(E/K)=1$.  
Together with Lemma~\ref{lem:rank-1}, it follows that ${\rm Sel}_\pp(K,T)=0$, and so ${\rm Reg}_{\pp,{\rm der},\gamma}=1$. 
Therefore, if  $s_\pp$ is any element of $\check{S}_p(E/K)\simeq E(K)\otimes\Z_p$ satisfying ${\rm res}_{\pp/{\rm tor}}(s_\pp)\neq 0$, Conjecture~\ref{conj:BSD}(i) now reads
\begin{equation}\label{eq:conj-rk0}
L_{\pp}^{\rm BDP}(0)\,\sim_p\,\biggl(\frac{1-a_p(E)+p}{p}\biggr)^2\cdot{\rm log}_\pp(s_\pp)^2\cdot\#\Sha(E/K)[p^\infty]\cdot\prod_{\ell\vert N}c_\ell^2.
\end{equation}
We can write $z_K\otimes 1=m\cdot s_\pp$, with $m\in\Z_p$ satisfying
\begin{equation}\label{eq:m}
{\rm ord}_p(m)={\rm ord}_p\bigl([E(K):\Z z_K]\bigr),\nonumber
\end{equation}
and formula \eqref{eq:bdp} can then be rewritten as 
\begin{equation}\label{eq:bdp-s}
L_{\pp}^{\rm BDP}(0)\,\sim_p\,\frac{m^2}{u_K^2c_E^2}\cdot\biggl(\frac{1-a_p(E)+p}{p}\biggr)^2\cdot{\rm log}_\pp(s_\pp)^2.
\end{equation}
Combining (\ref{eq:conj-rk0}) and (\ref{eq:bdp-s}), we thus see that Conjecture~\ref{conj:BSD}(i) is equivalent to 
\begin{equation}\label{eq:conclusion}
[E(K):\Z z_K]^2\,\sim_p\,u_K^2c_E^2\cdot\#\Sha(E/K)\cdot\prod_{\ell\mid N}c_\ell^2.
\end{equation}
By the Gross--Zagier formula, (\ref{eq:conclusion}) is equivalent to the $p$-part of the Birch--Swinnerton-Dyer formula for $L'(E/K,1)$ (see e.g. \cite[Lem.~10.1]{zhang-Kolyvagin}), so this concludes the proof.
\end{proof}







\section{Main result}\label{sec:main}

We shall say that the $p$-adic height pairing $h_\pp:=h_\pp^{(1)}$ is \emph{maximally non-degenerate} if, letting $e_i$ be as in \eqref{eq:ei}, we have
\[
e_i=\begin{cases}
\abs{ r^+-r^-}-1 &\textrm{if $i=2$,}\\
0&\textrm{if $i\geq 3$,}
\end{cases}
\]
where $r^\pm={\rm rank}_{\Z_p}(\check{S}_p(E/K)^\pm)$ for the $\pm$-eigenspace $\check{S}_p(E/K)^\pm$ of $\check{S}_p(E/K)$ under the action of complex conjugation.

 
The main result of this paper is the following. 

\begin{thm}\label{thm:A}
In the setting of $\S\ref{sec:BSDconj}$, let $F_{\ppbar}^{\rm BDP}\in\Lambda$ be a generator of ${\rm char}_\Lambda({\rm Sel}_{\ppbar}(K_\infty,W)^\vee)$.  
Put ${\varrho_{\rm alg}}:={\rm ord}_{J}F_{\ppbar}^{\rm BDP}$. 
 
\begin{itemize}
\item[(i)] Up to a $p$-adic unit,
\begin{align*}
 \frac{1}{\varrho_{\rm alg}!}\frac{d^{\varrho_{\rm alg}}}{dT^{\varrho_{\rm alg}}}F_{\ppbar}^{\rm BDP}\Bigl\vert_{T=0}=&\,\biggl(\frac{1-a_p(E)+p}{p}\biggr)^2\cdot{\rm log}_\pp(s_\pp)^2\\
 &\times{\rm Reg}_{\pp,{\rm der},\gamma}\cdot\#\Sha_{\rm BK}(K,W)\cdot\prod_{\ell\mid N}c_\ell^2.
\end{align*}
\item[(ii)] The following inequality holds
\[
{\varrho_{\rm alg}}\geq 2(\max\{r^+,r^-\}-1).
\]
Furthermore, equality is attained if and only if $h_\pp^{}$ is maximally non-degenerate.
\end{itemize}
\end{thm}

Combined with progress towards the Iwasawa--Greenberg Main Conjecture (Conjecture~\ref{conj:BDP-IMC}), we deduce the following result towards the $p$-adic Birch--Swinnerton-Dyer conjecture for $L_\pp^{\rm BDP}$. 

\begin{cor}\label{cor:B}
Let $\overline{\rho}:G_\Q\rightarrow{\rm Aut}_{\mathbf{F}_p}(E[p])$ be the mod $p$ representation associated with $E$.
If $p$ is ordinary, assume that:
\begin{itemize}
\item[(1a)] Either $N$ is square-free or there are at least two primes $\ell\Vert N$.
\item[(1b)] $\overline{\rho}$ is ramified at every prime $\ell\Vert N$.
\item[(1c)] $\overline{\rho}$ is surjective.
\item[(1d)] $a_p(E)\not\equiv 1\pmod{p}$.
\item[(1e)] $p>3$.
\end{itemize}
If $p$ is supersingular, assume that:
\begin{itemize}
\item[(2a)] $N$ is square-free.
\item[(2b)] $\overline{\rho}$ is ramified at every prime $\ell\mid N$.
\item[(2c)] Every prime above $p$ is totally ramified in $K_\infty/K$.
\item[(2d)] $p>3$, which implies $a_p(E)=0$.
\end{itemize}
Then,
\begin{itemize}
\item[(i)] The Leading Coefficient Formula of Conjecture~\ref{conj:BSD}(i) holds.
\item[(ii)] The following inequality holds 
\[
{\rm ord}_JL_\pp^{\rm BDP}\geq 2(\max\{r^+,r^-\}-1).
\] 
Equality is attained, and hence Conjecture~\ref{conj:BSD}(ii) holds, if and only if $h_\pp^{}$ is maximally non-degenerate.
\end{itemize}
\end{cor}

\begin{proof}
In the $p$-ordinary case, 
the Iwasawa--Greenberg Main Conjecture (Conjecture~\ref{conj:BDP-IMC}) was proved in \cite[Thm.\,B]{BCK} under hypotheses (1a)--(1e); in the $p$-supersingular case, a proof of the same conjecture under hypotheses (2a)--(2d) is given in Corollary~\ref{cor:BDP-IMC}. 
The result thus follows from Theorem~\ref{thm:A}.
\end{proof}


\section{Proof of Theorem~\ref{thm:A}}\label{sec:proofs}

Before delving into the details, we give a brief outline of the proof of the Leading Coefficient Formula in part (i) of Theorem~\ref{thm:A}. The much shorter proof of part (ii) is given in $\S\ref{subsec:order-van}$.

The proof is divided into four steps. 
%
%
%
Our starting point is Lemma~\ref{lem:well-known}, giving an expression for the leading coefficient at $T=0$ 
of the characteristic power series $F_X$ of a torsion $\Lambda$-module $X$, where $\Lambda=\Z_p\dBr{T}$. The formula is in terms of the orders of the flanking terms in the exact sequence
\[
0\rightarrow T^rX[T]\rightarrow T^rX\xrightarrow{T}T^rX\rightarrow T^rX/T^{r+1}X\rightarrow 0
\]
for any  $r\geq{\rm ord}_TF_X$. Together with the results from Section~\ref{sec:der-ht}, this lemma applied to $X_{\ppbar}={\rm Sel}_{\ppbar}(K_\infty,W)^\vee$ leads to a proof of the equality up to a $p$-adic unit
\begin{equation}\label{eq:step1-outline}
\frac{1}{\varrho_{\rm alg}!}\frac{d^{\varrho_{\rm alg}}}{dT^{\varrho_{\rm alg}}}F_{\ppbar}^{\rm BDP}\Bigl\vert_{T=0}\,\sim_p\,
\#\bigl(\underrightarrow{\mathfrak{S}}_{\ppbar}^{(r+1)}\bigr)\tag{Step 1}
\end{equation}
for any $r\geq\varrho_{\rm alg}$. Passing to the limit in $k$, the $\cO_k$-valued derived height pairings $h_{\pp,k}^{(i)}$ 
give rise to a collection of exact sequences
\begin{equation}\label{eq:4-term-refined}
0\rightarrow\underleftarrow{\mathfrak{S}}_{\pp}^{(i+1)}\rightarrow\underleftarrow{\mathfrak{S}}_{\pp}^{(i)}\xrightarrow{h_\pp^{(i)}}{\rm Hom}_{\Z_p}\bigl(\underrightarrow{\mathfrak{S}}_{\ppbar}^{(i)},\Q_p/\Z_p\bigr)\rightarrow{\rm Hom}_{\Z_p}\bigl(\underrightarrow{\mathfrak{S}}_{\ppbar}^{(i+1)},\Q_p/\Z_p\bigr)\rightarrow 0
\end{equation}
for $i\geq 1$, where $\underleftarrow{\mathfrak{S}}_{\qq}^{(i)}=\varprojlim_k\mathfrak{S}_{\qq,k}^{(i)}$ and $\underrightarrow{\mathfrak{S}}_{\qq}^{(i)}=\varinjlim_k\mathfrak{S}_{\qq,k}^{(i)}$. 

We find the derived regulator ${\rm Reg}_{\pp,{\rm der},\gamma}$ appearing naturally from an iterative computation using \eqref{eq:4-term-refined} for $i=1,\dots,r$, leading to a proof of the equality
\begin{equation}\label{eq:step2-outline}
\#\bigl(\underrightarrow{\mathfrak{S}}_{\ppbar}^{(r+1)}\bigr)\,\sim_p\,{\rm Reg}_{\pp,{\rm der},\gamma}\cdot\bigl[{\rm Sel}_\pp(K,T)\colon\underleftarrow{\mathfrak{S}}_{\pp}^{(1)}\bigr]\cdot
\#\bigl((\underrightarrow{\mathfrak{S}}_{\ppbar}^{(1)})_{/{\rm div}}\bigr)\tag{Step 2}
\end{equation}
for any $r\geq\varrho_{\rm alg}$. 

Next we study the local conditions cutting out the Selmer groups $\underleftarrow{\mathfrak{S}}_{\pp}^{(1)}$ and $\underrightarrow{\mathfrak{S}}_{\ppbar}^{(1)}$, arriving at a five-term exact sequence from which we can deduce the relation
\begin{equation}\label{eq:step3-outline}
\bigl[{\rm Sel}_\pp(K,T)\colon\underleftarrow{\mathfrak{S}}_{\pp}^{(1)}\bigr]\cdot
\#\bigl((\underrightarrow{\mathfrak{S}}_{\ppbar}^{(1)})_{/{\rm div}}\bigr)
=\#(E(\Q_p)[p^\infty])^2\cdot\prod_{w\mid N}c_w^{(p)}\cdot\#({\rm Sel}_{\ppbar}(K,W)_{/{\rm div}}).\tag{Step 3}
\end{equation}

Finally, using global duality and a computation using ${\rm log}_\pp$ we obtain
\begin{equation}\label{eq:step4-outline}
\#({\rm Sel}_{\ppbar}(K,W)_{/{\rm div}})\,\sim_p\,\biggl(\frac{1-a_p(E)+p}{p}\biggr)^2\cdot{\rm log}_\pp(s_\pp)^2\cdot\#\Sha_{\rm BK}(K,W)\cdot\frac{1}{(\#E(K_\pp)[p^\infty])^2},\tag{Step 4}
\end{equation} 
which in combination with the previous steps yields the Leading Coefficient Formula in part~(i) of Theorem~\ref{thm:A}. 


\subsection{Step 1: Non-semisimple torsion $\Lambda$-modules and augmentation filtration}\label{subsec:step1}



Let $X$ be a $\Lambda$-module, where $\Lambda=\Z_p\dBr{T}$. For every $r\geq 0$, denote by 
\[
\beta_X^{(r)}:T^{r}X\rightarrow T^rX
\]
the map given by multiplication by $T$, and put $h(\beta_X^{(r)})=\#{\rm coker}(\beta_X^{(r)})/\#{\rm ker}(\beta_X^{(r)})$ whenever both terms in the right-hand side are finite.

\begin{lem}\label{lem:well-known}
Let $X$ be a finitely generated torsion $\Lambda$-module, and let $F_X\in\Lambda$ be a generator of ${\rm char}_\Lambda(X)$. 
Put
\[
\varrho_X:={\rm ord}_{T}F_X.
\]
For any $r\geq\varrho_X$, the sub-quotients $T^rX/T^{r+1}X$ and $T^rX[T]$ are both finite.
In addition,
\[
\frac{1}{\varrho_X!}\frac{d^{\varrho_X}}{dT^{\varrho_X}}F_X\Bigl\vert_{T=0}\;\sim_p\;h(\beta_X^{(r)})=
\frac{\#\bigl(T^rX/T^{r+1}X\bigr)}{\#(T^rX[T])},
\] 
where $\sim_p$ denotes equality up to a $p$-adic unit.
\end{lem}

\begin{proof}
Suppose first that $X$ is an elementary module, in the sense that
\[
X=\Lambda/(f),
\]
where $f=a_nT^n+a_{n+1}T^{n+1}+\cdots\in \Lambda$ with $a_n\neq 0$ (so $n=\varrho_X$). Then for any $r\geq n$ we find
\begin{align*}
X[T^r]&=(T^{-n}f)/(f),\\
T^rX&=(T^r, f)/(f)\simeq (T^r)/(T^r)\cap (f)=(T^r)/(T^{r-n}f),
\end{align*}
Therefore,for any $r\geq\varrho_X$ we have that $T^rX[T]\simeq X[T^{r+1}]/X[T^r]$ is trivial and  
\[
T^rX/T^{r+1}X\simeq (T^r)/(T^{r+1}, T^{r-n}f)\simeq \Z_p/(a_n)
\] 
so the result is true in this case.

In general, by the structure theorem 
there exists a $\Lambda$-module homomorphism 
\[
\phi:X\rightarrow Y 
\]
with finite kernel and cokernel, where $Y$ is a direct sum of elementary modules as above. Since $X$ and $Y$ have the same characteristic ideal and by the above argument the result is true for $Y$, it remains to show that $h(\beta^{(r)}_X)=h(\beta^{(r)}_Y)$ for any $r\geq \varrho_X=\varrho_Y$.

For any $\Lambda$-module homomorphism $\beta:M\rightarrow M$ with finite kernel and cokernel, put $h(\beta)=\#{\rm coker}(\beta)/\#{\rm ker}(\beta)$. 
%
Note that if $0\rightarrow A\rightarrow B\rightarrow C\rightarrow 0$ is a $\Lambda$-module exact sequence, and any two of the multiplication-by-$T$ maps $T_A:A\rightarrow A$, $T_B:B\rightarrow B$, $T_C:C\rightarrow C$ have  finite kernel and cokernel, then from an easy application of the snake lemma we see that $h(T_A)$, $h(T_B)$, $h(T_C)$ are all defined, with
\begin{equation}\label{eq:mult-h}
h(T_B)=h(T_A)\cdot h(T_C).
\end{equation}

For any $r>0$, $\phi$ induces maps $\phi_r:T^rX\rightarrow T^rY$ with finite kernel and cokernel. 
Applying \eqref{eq:mult-h} to the tautological exact sequence $0\rightarrow{\rm im}(\phi_r)\rightarrow T^rY\rightarrow{\rm coker}(\phi_r)\rightarrow 0$, we obtain
\begin{equation}\label{eq:star-1}
h(\beta_{Y}^{(r)})=h(T_{{\rm im}(\phi_r)})\cdot h(T_{{\rm coker}(\phi_r)})=h(T_{{\rm im}(\phi_r)})
\end{equation}
for any $r\geq\varrho_X$, using that $h(T_{{\rm coker}(\phi_r)})=1$  (since $\#{\rm coker}(\phi_r)<\infty$) for the last equality. On the other hand, applied to $0\rightarrow{\rm ker}(\phi_r)\rightarrow T^rX\rightarrow{\rm im}(\phi_r)\rightarrow 0$, \eqref{eq:mult-h} similarly gives
\begin{equation}\label{eq:star-2}
h(\beta_X^{(r)})=h(T_{{\rm ker}(\phi_r)})\cdot h(T_{{\rm im}(\phi_r)})=h(T_{{\rm im}(\phi_r)})
\end{equation}
for any $r\geq\varrho_X$. Combining \eqref{eq:star-1} and \eqref{eq:star-2}, the result follows.
\end{proof}

\begin{rem}
When the action of $T$ on $X$ is ``semi-simple'' (i.e., when up to pseudo-isomorphism $X$ is of the form $\oplus_i\Lambda/(f_i)$ with $f_i\in\Lambda$ satisfying ${\rm ord}_{T}(f_i)\leq 1$ for all $i$), Lemma~\ref{lem:well-known} recovers a well-known result (see e.g. \cite[$\S{1.4}$, Lemme]{PR-ht-abvar}).  
\end{rem}


For the rest of this section, we let $\Lambda=\Z_p\dBr{\Gamma}$ be the anticyclotomic Iwasawa algebra and $J\subset \Lambda$ the augmentation ideal. 
We shall often identify $\Lambda$ (resp. $J$) with the one variable power series ring $\Z_p\dBr{T}$ (resp. $(T)$) setting $T=\gamma-1$ for a fixed choice of topological generator $\gamma\in\Gamma$. 


\begin{prop}\label{prop:well-known}
Let $\qq\in\{\pp,\ppbar\}$. 
Let $F_\qq^{\rm BDP}\in\Lambda$ be a generator of the characteristic ideal of $X_\qq={\rm Sel}_\qq(K,S_\infty)^\vee$, and put ${\varrho_{\rm alg}}={\rm ord}_{J}F_\qq^{\rm BDP}$. Then
\[
\frac{1}{\varrho_{\rm alg}!}\frac{d^{\varrho_{\rm alg}}}{dT^{\varrho_{\rm alg}}}F_\qq^{\rm BDP}\Bigl\vert_{T=0}\,\sim_p\,\#\bigl(\underrightarrow{\mathfrak{S}}_{\qq}^{(r+1)}\bigr)
\]
for any $r\geq{\varrho_{\rm alg}}$.
\end{prop}

\begin{proof}
After Lemma~\ref{lem:well-known}, it suffices to show that we have
\begin{align}\label{eq:reduction}
\#(J^rX_\qq/J^{r+1}X_\qq)=\#\bigl(\underrightarrow{\mathfrak{S}}_{\qq}^{(r+1)}\bigr),\qquad J^rX_\qq[J]=0.
\end{align}
for any $r\geq{\varrho_{\rm alg}}$. Consider the exact sequence
\[
0\rightarrow\frac{\mathfrak{Sel}_\qq(K,S_\infty)[J^{r+1}]}{\mathfrak{Sel}_\qq(K,S_\infty)[J^r]}\rightarrow\frac{\mathfrak{Sel}_\qq(K,S_\infty)}{\mathfrak{Sel}_\qq(K,S_\infty)[J^r]}\rightarrow\frac{\mathfrak{Sel}_\qq(K,S_\infty)}{\mathfrak{Sel}_\qq(K,S_\infty)[J^{r+1}]}\rightarrow 0,
\]
Taking Pontryagin duals and noting that 
\[
\bigl(\mathfrak{Sel}_\qq(K,S_\infty)/\mathfrak{Sel}_\qq(K,S_\infty)[J^r]\bigr)^\vee=\bigl(J^r\mathfrak{Sel}_\qq(K,S_\infty)\bigr)^\vee=J^rX_\qq,
\] 
using Corollary~\ref{cor:control} for the last equality, we obtain the exact sequence
\begin{equation}\label{eq:duals}
0\rightarrow J^{r+1}X_\qq\rightarrow J^rX_\qq\rightarrow{\rm Hom}_{\Z_p}\biggl(\frac{\mathfrak{Sel}_\qq(K,S_\infty)[J^{r+1}]}{\mathfrak{Sel}_\qq(K,S_\infty)[J^r]},\Q_p/\Z_p\biggr)\rightarrow 0.
\end{equation}
Via 
the maps $\phi_{r+1,\gamma}$ in \eqref{eq:phi-gamma} given by multiplication by $(\gamma-1)^{r}$,  
the last term in this sequence is identified with the Pontryagin dual of 
$J^r\mathfrak{Sel}_\qq(K,S_\infty)[J]\simeq\underrightarrow{\mathfrak{S}}_\qq^{(r+1)}$, and so the first equality in \eqref{eq:reduction} (for any $r\geq 0$) follows from \eqref{eq:duals}.

On the other hand, since by Proposition~\ref{prop:Lambda-tors} we know that $X_\qq$ is $\Lambda$-torsion,  Corollary~\ref{cor:der-ht}(ii) and our running hypothesis \eqref{eq:h0-intro} imply that the filtration \eqref{eq:fil-T} satisfies $\underleftarrow{\mathfrak{S}}^{(i)}_{\qq}=0$ for $i\gg 0$. Let 
\[
i_0=\max\{i\geq 1\,\colon\,\underleftarrow{\mathfrak{S}}^{(i)}_\qq\neq 0\}.
\]
%
Then by Corollary~\ref{cor:der-ht} we may fix a $\Lambda$-module pseudo-isomorphism
\begin{equation}\label{eq:pseudo-iso}
X_\qq\,\sim(\Lambda/J)^{e_1}\oplus(\Lambda/J^2)^{e_2}\oplus\cdots\oplus(\Lambda/J^{i_0})^{e_{i_0}}\oplus M
\end{equation}
with $M$ a torsion $\Lambda$-module with characteristic ideal prime to $J$, and therefore $\#(M[J])<\infty$. In particular
\begin{equation}\label{eq:varrho}
{\varrho_{\rm alg}}=e_1+2e_2+\cdots+i_0e_{i_0},
\end{equation}
which shows that $r\geq i_0$ for any $r\geq{\varrho_{\rm alg}}$. 
From \eqref{eq:pseudo-iso}, we thus see that for any $r\geq{\varrho_{\rm alg}}$ (in fact, $r\geq i_0$ suffices), the $\Lambda$-submodule $J^rX_\qq[J]$ of $X_\qq$ is finite, and so by Proposition~\ref{prop:greenberg}  the second equality in \eqref{eq:reduction} follows, whence the result.
\end{proof}

\subsection{Step 2: Derived $p$-adic regulator}\label{subsec:step2} 


We begin with a basic algebraic lemma.


\begin{lem}\label{lem:alg}
Let $0\rightarrow A\xrightarrow{h}B\rightarrow C\rightarrow 0$ be an exact sequence of finitely generated modules over $\Z_p$. 
Then
\[
\#(A_{\rm tor})\cdot\#(C_{\rm tor})\,\sim_p\,{\rm det}^*h\cdot\#(B_{\rm tor}),
\]
where ${\rm det}^*h$ is the product of the non-zero entries in the Smith normal form of $A_{/{\rm tor}}\xrightarrow{h}B_{/{\rm tor}}$.
\end{lem}

\begin{proof}
This follows upon noting the relations 
$\#(C_{\rm tor})=\#\bigl(B_{/{\rm tor}}/h(A_{/{\rm tor}})\bigr)_{\rm tor}\cdot\#(B_{\rm tor}/h(A_{\rm tor}))$ and $\#\bigl(B_{/{\rm tor}}/h(A_{/{\rm tor}})\bigr)_{\rm tor}\sim_p{\rm det}^*h$. 
\end{proof}

For $\qq\in\{\pp,\ppbar\}$, put 
\[
\mathfrak{Sel}_\qq(K,T):=\underleftarrow{\mathfrak{S}}_\qq^{(1)}=\varprojlim_k\mathfrak{Sel}_\qq(K,S_\infty^{(k)})[J],\qquad\mathfrak{Sel}_{\qq}(K,W):=\underrightarrow{\mathfrak{S}}_{\qq}^{(1)}.
\]
In particular, we have seen in Corollary~\ref{cor:control-mazur} that $\mathfrak{Sel}_\qq(K,T)$ is contained in the $\qq$-strict Selmer groups ${\rm Sel}_\qq(K,T)$ with finite index.

\begin{prop}\label{prop:recursion}
For any $r\geq\varrho_{\rm alg}:={\rm ord}_JF_{\ppbar}^{\rm BDP}$, we have
\[
\#\bigl(\underrightarrow{\mathfrak{S}}_{\ppbar}^{(r+1)}\bigr)\,\sim_p\,{\rm Reg}_{\pp,{\rm der},\gamma}\cdot\bigl[{\rm Sel}_\pp(K,T)\colon\mathfrak{Sel}_\pp(K,T)\bigr]\cdot
\#\bigl(\mathfrak{Sel}_{\ppbar}(K,W)_{/{\rm div}}\bigr).
\]
\end{prop}

\begin{proof}
Passing to the limit in $k$, the $\cO_k$-valued $i$-th derived height pairings $h_{\pp,k}^{(i)}$ of Theorem~\ref{thm:howard-ht-pk} induce a pairing
\[
     \underleftarrow{\mathfrak{S}}_{\pp}^{(i)}\times\underrightarrow{\mathfrak{S}}_{\ppbar}^{(i)}\rightarrow\Q_p/\Z_p.
\]
By the descriptions of the kernels of $h_{\pp,k}^{(i)}$ in Theorem~\ref{thm:howard-ht-pk}, this gives rise to the exact sequence
\begin{equation*}
0\rightarrow\underleftarrow{\mathfrak{S}}_{\pp}^{(i)}/\underleftarrow{\mathfrak{S}}_{\pp}^{(i+1)}\xrightarrow{\alpha^{(i)}}{\rm Hom}_{\Z_p}\bigl(\underrightarrow{\mathfrak{S}}_{\ppbar}^{(i)},\Q_p/\Z_p\bigr)\rightarrow{\rm Hom}_{\Z_p}\bigl(\underrightarrow{\mathfrak{S}}_{\ppbar}^{(i+1)},\Q_p/\Z_p\bigr)\rightarrow 0.
\end{equation*}
We also have the tautological exact sequence
\[
0\rightarrow\underleftarrow{\mathfrak{S}}_{\pp}^{(i)}/\underleftarrow{\mathfrak{S}}_{\pp}^{(i+1)}\xrightarrow{\beta^{(i)}}{\rm Sel}_\pp(K,T)/\underleftarrow{\mathfrak{S}}_{\pp}^{(i+1)}\rightarrow{\rm Sel}_\pp(K,T)/\underleftarrow{\mathfrak{S}}_{\pp}^{(i)}\rightarrow 0.
\]
Applying Lemma~\ref{lem:alg} to the above two exact sequences gives
\begin{align*}
\frac{\#\bigl((\underrightarrow{\mathfrak{S}}_{\ppbar}^{(i+1)})_{/{\rm div}}\bigr)}{\#\bigl((\underrightarrow{\mathfrak{S}}_{\ppbar}^{(i)})_{/{\rm div}}\bigr)}
&\sim_p
\frac{\det^*\alpha^{(i)}}{\#\bigl((\underleftarrow{\mathfrak{S}}_{\pp}^{(i)}/\underleftarrow{\mathfrak{S}}_{\pp}^{(i+1)})_{\rm tor}\bigr)},\\
{\#\bigl((\underleftarrow{\mathfrak{S}}_{\pp}^{(i)}/\underleftarrow{\mathfrak{S}}_{\pp}^{(i+1)})_{\rm tor}\bigr)}
&\sim_p
\frac{\det^*\beta^{(i)}\cdot \#\bigl({\rm Sel}_\pp(K,T)/\underleftarrow{\mathfrak{S}}_{\pp}^{(i+1)}\bigr)_{\rm tor}}{\#\bigl({\rm Sel}_\pp(K,T)/\underleftarrow{\mathfrak{S}}_{\pp}^{(i)}\bigr)_{\rm tor}},
\end{align*}
from which we get
\begin{equation}\label{eq:rel-i}
\frac{\#\bigl((\underrightarrow{\mathfrak{S}}_{\ppbar}^{(i+1)})_{/{\rm div}}\bigr)}{\#\bigl((\underrightarrow{\mathfrak{S}}_{\ppbar}^{(i)})_{/{\rm div}}\bigr)}
\sim_p
\frac{\det^*\alpha^{(i)}}{\det^*\beta^{(i)}}\cdot \frac{\#\bigl({\rm Sel}_\pp(K,T)/\underleftarrow{\mathfrak{S}}_{\pp}^{(i)}\bigr)_{\rm tor}}{\#\bigl({\rm Sel}_\pp(K,T)/\underleftarrow{\mathfrak{S}}_{\pp}^{(i+1)}\bigr)_{\rm tor}}.
\end{equation}
For any $r\geq\varrho_{\rm alg}$, using that $\#(\underrightarrow{\mathfrak{S}}_{\ppbar}^{(r+1)})<\infty$ (and so $\underleftarrow{\mathfrak{S}}_{\pp}^{(r+1)}=0$) and taking the product of \eqref{eq:rel-i} for $i=1,\dots,r$, we arrive at
\begin{equation}\label{eq:alphabeta}
\begin{aligned}
   \#\bigl(\underrightarrow{\mathfrak{S}}_{\ppbar}^{(r+1)}\bigr)
   &= \#\bigl((\underrightarrow{\mathfrak{S}}_{\ppbar}^{(r+1)})_{/{\rm div}}\bigr)\\
   &\sim_p \prod_{i=1}^{\varrho_{\rm alg}} \frac{\det^*\alpha^{(i)}}{\det^*\beta^{(i)}}
   \cdot  \frac{\#\bigl({\rm Sel}_\pp(K,T)/\underleftarrow{\mathfrak{S}}_{\pp}^{(1)}\bigr)_{\rm tor}}{\#\bigl({\rm Sel}_\pp(K,T)/\underleftarrow{\mathfrak{S}}_{\pp}^{(r+1)}\bigr)_{\rm tor}}
   \cdot \#\bigl((\underrightarrow{\mathfrak{S}}_{\ppbar}^{(1)})_{/{\rm div}}\bigr)\\
   &\sim_p \prod_{i=1}^{\varrho_{\rm alg}} \frac{\det^*\alpha^{(i)}}{\det^*\beta^{(i)}}
   \cdot\bigl[{\rm Sel}_\pp(K,T)\colon\mathfrak{Sel}_\pp(K,T)\bigr]\cdot \#\bigl((\underrightarrow{\mathfrak{S}}_{\ppbar}^{(1)})_{/{\rm div}}\bigr).
\end{aligned}
\end{equation}
Thus we are reduced to showing that 
\begin{equation}\label{eq:ab-reg}
\frac{\det^*\alpha^{(i)}}{\det^*\beta^{(i)}}=R^{(i)}_{\pp,\gamma},
\end{equation}
where $R^{(i)}_{\pp,\gamma}$ is the partial regulator in Definition~\ref{def:der-Reg}. 

By Corollary~\ref{cor:control}, we note that 
Proposition~\ref{prop:control-mazur} 
amounts to the statement that ${\rm Sel}_{\ppbar}(K,W)$ is contained in $\underrightarrow{\mathfrak{S}}_{\ppbar}^{(1)}\simeq\mathfrak{Sel}_{\ppbar}(K,S_\infty)[J]$ with finite index. Therefore, $(\underrightarrow{\mathfrak{S}}_{\ppbar}^{(1)})_{\rm div}={\rm Sel}_{\ppbar}(K,W)_{\rm div}$.
It follows from the definition of ${\rm Sel}_{\ppbar}(K,T)$ and ${\rm Sel}_{\ppbar}(K,W)$ by propagating  the local conditions $\rH^1_{\ppbar}(K_w,V)$ that 
\begin{equation}\label{eq:div}
{\rm Sel}_{\ppbar}(K,T)\otimes_{\Z_p}\Q_p/\Z_p={\rm Sel}_{\ppbar}(K,W)_{\rm div}.    
\end{equation}
From this, it follows that for every $i\geq 1$ we have 
\begin{equation}\label{eq:saturation}
\underleftarrow{\overline{\mathfrak{S}}}_{\ppbar}^{(i)}\otimes\Q_p/\Z_p\simeq(\underrightarrow{\mathfrak{S}}_{\ppbar}^{(i)})_{\rm div}.
\end{equation}
(Indeed, because $\underleftarrow{\overline{\mathfrak{S}}}_{\ppbar}^{(i)}$ is defined as the saturation of $\underleftarrow{\mathfrak{S}}_{\ppbar}^{(i)}$ inside  ${\rm Sel}_{\ppbar}(K,T)$, the composite map
\[
\underleftarrow{\overline{\mathfrak{S}}}_{\ppbar}^{(i)}\otimes\Q_p/\Z_p\rightarrow{\rm Sel}_{\ppbar}(K,T)\otimes\Q_p/\Z_p\simeq(\underrightarrow{\mathfrak{S}}_{\ppbar}^{(1)})_{\rm div}
\]
is injective with image $\underrightarrow{\mathfrak{S}}_{\ppbar}^{(i)})_{\rm div}$.) 
From \eqref{eq:saturation} we deduce  
\begin{equation}\label{eq:dual-dual}
{\rm Hom}_{\Z_p}\bigl((\underrightarrow{\mathfrak{S}}_{\ppbar}^{(i)})_{\rm div},\Q_p/\Z_p\bigr)\simeq{\rm Hom}_{\Z_p}(\underleftarrow{\overline{\mathfrak{S}}}_{\ppbar}^{(i)},\Z_p),
\end{equation}
for all $i\geq 1$. Note that the composition
\[
   \underleftarrow{\mathfrak{S}}_{\pp}^{(i)}/\underleftarrow{\mathfrak{S}}_{\pp}^{(i+1)}\xrightarrow{\alpha^{(i)}}{\rm Hom}_{\Z_p}\bigl(\underrightarrow{\mathfrak{S}}_{\ppbar}^{(i)},\Q_p/\Z_p\bigr)\rightarrow{\rm Hom}_{\Z_p}\bigl((\underrightarrow{\mathfrak{S}}_{\ppbar}^{(i)})_{\rm div},\Q_p/\Z_p\bigr)\simeq{\rm Hom}_{\Z_p}(\underleftarrow{\overline{\mathfrak{S}}}_{\ppbar}^{(i)},\Z_p)
\]
is induced by the pairing $h^{(i)}_\pp$. Hence,
\begin{align*}
   \det\nolimits^*\alpha^{(i)}&={\rm disc}\Bigl(h_\pp^{(i)}\vert_{\bigl(\underleftarrow{\mathfrak{S}}_{\pp}^{(i)}/\underleftarrow{\mathfrak{S}}_{\pp}^{(i+1)}\bigr)_{/{\rm tor}}\times\underleftarrow{\overline{\mathfrak{S}}}_{\ppbar}^{(i)}/\underleftarrow{\overline{\mathfrak{S}}}_{\ppbar}^{(i+1)}}\Bigr)\\
&=R_{\pp,\gamma}^{(i)}\cdot\biggl[\underleftarrow{\overline{\mathfrak{S}}}_{\pp}^{(i)}/\underleftarrow{\overline{\mathfrak{S}}}_{\pp}^{(i+1)}\colon\bigl(\underleftarrow{\mathfrak{S}}_{\pp}^{(i)}/\underleftarrow{\mathfrak{S}}_{\pp}^{(i+1)}\bigr)_{/{\rm tor}}\biggr]\\
&=R_{\pp,\gamma}^{(i)}\cdot\det\nolimits^*\beta^{(i)}.
\end{align*}
This shows \eqref{eq:ab-reg}, which together with \eqref{eq:alphabeta} yields the result.
\end{proof}


\subsection{Step 3: Local universal norms}\label{subsec:step3}

Note that by definition, we have
\[
\mathfrak{Sel}_\qq(K,T)=\varprojlim_k\mathfrak{Sel}_\qq(K,S^{(k)}),\qquad \mathfrak{Sel}_{\qq}(K,W)=\varinjlim_k\mathfrak{Sel}_{\qq}(K,S^{(k)}),
\]
where $\mathfrak{Sel}_{\qq}(K,S^{(k)})$ is as in Definition~\ref{def:Sel-k}.
 



\begin{prop}\label{prop:self-dual}
For 
every $w\in\Sigma_f$, the local conditions $\rH^1_{\mathcal{F}_\qq}(K_w,S^{(k)})$ and $\rH^1_{\mathcal{F}_{\overline{\qq}}}(K_w,S^{(k)})$ are  exact orthogonal complements under the local Tate pairing
\begin{equation}\label{eq:local-Tate-Sk}
\rH^1(K_w,S^{(k)})\times\rH^1(K_w,S^{(k)})\rightarrow\cO_k
\end{equation}
induced by the Weil pairing $e:S^{(k)}\times S^{(k)}\rightarrow\cO_k(1)$.
\end{prop}

\begin{proof}
We begin by considering the case $w=\qq$. From the definitions, we have
\begin{equation}\label{eq:q-qbar}
\begin{aligned}
\rH^1_{\mathcal{F}_\qq}(K_\qq,S^{(k)})&={\rm im}\bigl\{\rH^0(K_\qq,S_\infty^{(k)})_\Gamma\hookrightarrow\rH^1(K_\qq,S^{(k)})\bigr\},\\
\rH^1_{\mathcal{F}_{\qqbar}}(K_\qq,S^{(k)})&={\rm im}\bigl\{\rH^1(K_\qq,S_{\rm Iw}^{(k)})_\Gamma\hookrightarrow\rH^1(K_\qq,S^{(k)})\bigr\}.
\end{aligned}
\end{equation}
To see the second equality, we note that the propagation to $S^{(k)}$ of 
\begin{align*}
\rH^1_{\mathcal{F}_{\qqbar}}(K_\qq,S_\infty^{(k)})&:={\rm im}\bigl\{\rH^1(K_\qq,S_{\mathcal{K}}^{(k)})\rightarrow\rH^1(K_\qq,S_\infty^{(k)})\bigr\}={\rm ker}\{\rH^1(K_\qq,S_\infty^{(k)})\rightarrow\rH^2(K_\qq,S_{\rm Iw}^{(k})\}
\end{align*}
equals the kernel of the composition $\rH^1(K_\qq,S^{(k)})\rightarrow\rH^1(K_\qq,S_\infty^{(k)})\rightarrow\rH^2(K_\qq,S_{\rm Iw}^{(k)})$, which is the same as the connecting homomorphism for the exact sequence
\begin{equation}\label{eq:SIwext}
   0\rightarrow S_{\rm Iw}^{(k)}\xrightarrow{\gamma-1} S_{\rm Iw}^{(k)}\rightarrow S^{(k)}\rightarrow 0.
\end{equation} Hence,
\[
    \rH^1_{\mathcal{F}_{\qqbar}}(K_\qq,S_\infty^{(k)})={\rm ker}\{\rH^1(K_\qq,S^{(k)})\rightarrow\rH^2(K_\qq,S_{\rm Iw}^{(k)})\}={\rm im}\bigl\{\rH^1(K_\qq,S_{\rm Iw}^{(k)})_\Gamma\hookrightarrow\rH^1(K_\qq,S^{(k)})\bigr\}.
\]
Now, since we have 
\begin{equation}\label{eq:iw-infty}
\frac{\rH^1(K_{\qq},S^{(k)})}{\rH^1(K_\qq,S_{\rm Iw}^{(k)})_\Gamma}\simeq\rH^2(K_\qq,S_{\rm Iw}^{(k)})[J]\simeq\bigl(\rH^0(K_\qq,S_\infty^{(k)})_\Gamma\bigr)^\vee
\end{equation}
from the cohomology long exact sequence associated with \eqref{eq:SIwext} and Tate's local duality, the orthogonality assertion for $w=\qq$ follows. The argument for $w=\qqbar$ is the same.

Next we consider the case $w\in\Sigma_f\setminus\{\pp,\ppbar\}$. In this case, $\rH^1_{\mathcal{F}_\qq}(K_w,S_\infty^{(k)})=\rH^1_{\mathcal{F}_{\qqbar}}(K_w,S_\infty^{(k)})=0$, 
 and therefore
\begin{align*}
\rH^1_{\mathcal{F}_\qq}(K_w,S^{(k)})=\rH^1_{\mathcal{F}_{\qqbar}}(K_w,S^{(k)})&=\ker\bigl\{\rH^1(K_w,S^{(k)})\rightarrow\rH^1(K_w,S_\infty^{(k)})\bigr\}\\
&={\rm im}\bigl\{\rH^0(K_w,S_\infty^{(k)})_\Gamma\hookrightarrow\rH^1(K_w,S^{(k)})\bigr\}.
\end{align*}
From the short exact sequence
\begin{equation}\label{eq:ses-Iw}
0\rightarrow\rH^1(K_w,S_{\rm Iw}^{(k)})_\Gamma\rightarrow\rH^1(K_w,S^{(k)})\rightarrow\rH^2(K_w,S_{\rm Iw}^{(k)})[J]\simeq\bigl(\rH^0(K_w,S_\infty^{(k)})_\Gamma\bigr)^\vee
\rightarrow 0
\end{equation}
induced by \eqref{eq:SIwext} and local Tate duality, we see that 
\[
\rH^1_{\mathcal{F}_\qq}(K_w,S^{(k)})=\bigl({\rm im}\{\rH^1(K_w,S_{\rm Iw}^{(k)})_\Gamma\hookrightarrow\rH^1(K_w,S^{(k)})\}\bigr)^\perp,
\]
where the superscript $\perp$ denotes the orthogonal complement under  \eqref{eq:local-Tate-Sk}. 
Thus it suffices to establish the equality
\begin{equation}\label{eq:h0-h1}
\rH^0(K_w,S_\infty^{(k)})_\Gamma=\rH^1(K_w,S_{\rm Iw}^{(k)})_\Gamma
\end{equation}
inside $\rH^1(K_w,S^{(k)})$. Consider the commutative diagram with exact rows
\[
\xymatrix{
0\ar[r]&\rH^0(K_w,S_\infty^{(k)})_{/\cK}\ar[r]\ar[d]^{\gamma-1}&\rH^1(K_w,S_{\rm Iw}^{(k)})\ar[r]\ar[d]^{\gamma-1}&{\rm im}\{\rH^1(K_w,S_{\rm Iw}^{(k)})\rightarrow\rH^1(K_w,S_{\cK}^{(k)})\}\ar[r]\ar[d]^{\gamma-1}&0\\
0\ar[r]&\rH^0(K_w,S_\infty^{(k)})_{/\cK}\ar[r]&\rH^1(K_w,S_{\rm Iw}^{(k)})\ar[r]&{\rm im}\{\rH^1(K_w,S_{\rm Iw}^{(k)})\rightarrow\rH^1(K_w,S_\cK^{(k)})\}\ar[r]&0,
}
\]
where the rows are induced by \eqref{eq:S:K-P} and the subscript $/\cK$ denotes the quotient by the natural image of $\rH^0(K_w,S_\cK^{(k)})$ in $\rH^0(K_w,S_\infty^{(k)})$. Using the fact that multiplication by $\gamma-1$ is invertible in $S_\cK^{(k)}$, the snake lemma applied to this diagram yields an injection 
\begin{equation}\label{eq:incl-h0-h1}
\rH^0(K_w,S_\infty^{(k)})_\Gamma\hookrightarrow\rH^1(K_w,S_{\rm Iw}^{(k)})_\Gamma.
\end{equation}
On the other hand, we can compute
\begin{equation}\label{eq:order}
\#\rH^0(K_w,S_\infty^{(k)})_\Gamma=\#\Bigl(\bigoplus_{\eta\vert w}E(K_{\infty,\eta})[p^k]\Bigr)_\Gamma=\#\Bigl(\bigoplus_{\eta\vert w}E(K_{\infty,\eta})[p^k]\Bigr)[J]=\#E(K_w)[p^k],
\end{equation}
using the finiteness of $E(K_{\infty,\eta})[p^k]$ and the fact that $w$ is finitely decomposed in $K_\infty/K$ (since it splits in $K$) for the second equality. From \eqref{eq:ses-Iw}, we similarly find
\begin{equation}\label{eq:index-Iw}
\bigl[\rH^1(K_w,S^{(k)}):\rH^1(K_w,S_{\rm Iw}^{(k)})_\Gamma\bigr]=\#\rH^2(K_w,S_{\rm Iw}^{(k)})[J]=\#E(K_w)[p^k].
\end{equation}
Since $\#\rH^1(K_w,S^{(k)})=(\#E(K_w)[p^k])^2$ by Tate's local Euler characteristic formula and Tate's local duality, the desired equality \eqref{eq:h0-h1} now follows from the combination of \eqref{eq:incl-h0-h1}, \eqref{eq:order}, and \eqref{eq:index-Iw}, thereby concluding the proof.
\end{proof}


From the analysis in the proof of Proposition~\ref{prop:self-dual}, we deduce the following two key results.

\begin{cor}\label{cor:control-1}
We have
\[
\mathfrak{Sel}_\qq(K,T)\simeq{\rm ker}\Biggl\{\rH^1(G_{K,\Sigma},T)\rightarrow\rH^1(K_\qq,T)\times\prod_{w\in\Sigma_f\setminus\{\qq\}}\frac{\rH^1(K_w,T)}{\rH^1(K_w,T)^u}\Biggr\},
\]
where $\rH^1(K_w,T)^u:=\varprojlim_k{\rm im}\bigl\{\rH^1(K_w,S_{\rm Iw}^{(k)})_\Gamma\rightarrow\rH^1(K_w,S^{(k)})\bigr\}$.
\end{cor}

\begin{proof}
By \eqref{eq:q-qbar} and the equality \eqref{eq:h0-h1} established in the proof of Proposition~\ref{prop:self-dual}, the local conditions cutting out $\mathfrak{Sel}_\qq(K,T)$ are given by $\rH^1(K_w,T)^u$ for $w\in\Sigma_f\setminus\{\qq\}$. 
Restricted to $\qq$, the classes in $\mathfrak{Sel}_\qq(K,T)$ land in the image of $\varprojlim_k\rH^0(K_{\qq},S_\infty^{(k)})_\Gamma$ in $\rH^1(K_\qq,T)$, and so their restriction to $\qq$ vanishes by the finiteness of $\bigoplus_{\eta\mid \qq}E(K_{\infty,\eta})[p^\infty]$. 
\end{proof}


\begin{cor}\label{cor:control-2}
We have a five-term exact sequence
\[
0\rightarrow\mathfrak{Sel}_\qq(K,T)\rightarrow{\rm Sel}_\qq(K,T)\rightarrow\mathcal{L}\rightarrow\mathfrak{Sel}_{\qqbar}(K,W)^\vee\rightarrow{\rm Sel}_{\overline\qq}(K,W)^\vee\rightarrow 0
\]
with
\[
\#\mathcal{L}=(\#E(\Q_p)[p^\infty])^2\cdot\prod_{w\mid N}c_w^{(p)},
\]
where $c_w^{(p)}$ is the $p$-part of the Tamagawa number of $E/K_w$.
\end{cor}

\begin{proof}
%

In light of Proposition~\ref{prop:self-dual} and Corollary~\ref{cor:control-1}, Poitou--Tate duality gives rise to a five-term exact sequence as in the statement with
\begin{equation}\label{eq:def-L}
\mathcal{L}=\rH^1_\qq(K_\qq,T)\times\prod_{w\in\Sigma_f\setminus\{\qq\}}\frac{\rH^1_\qq(K_w,T)}{\rH^1(K_w,T)^u}.
\end{equation}
It remains to compute the size of each the factors in the right-hand side.

For the first factor, Definition~\ref{def:BDP-Sel} gives $\rH^1_\qq(K_\qq,T)=\rH^1(K_\qq,T)_{\rm tor}$, and from the combination of \eqref{eq:q-tor} and \eqref{eq:q-tor-bis} we have
\begin{equation}\label{eq:order-q}
\#\rH^1_\qq(K_\qq,T)=\#\rH^1(K_\qq,T)_{\rm tor}=\#E(\Q_p)[p^\infty].
\end{equation}
For $w=\qqbar$, the numerator in the second factor is $\rH^1_\qq(K_{\qqbar},T)=\rH^1(K_{\qqbar},T)$, 
and from \eqref{eq:iw-infty} we have
\[
\rH^1(K_{\qqbar},T)/\rH^1(K_{\qqbar},T)^u\simeq\biggl(\varinjlim_k\rH^0(K_{\qqbar},S_\infty^{(k)})_\Gamma\biggr)^\vee.
\]
Since $\varinjlim_k\rH^0(K_{\qqbar},S_\infty^{(k)})$ is finite by \cite[Lem.~2.7]{KO}, it follows that
\begin{equation}\label{eq:order-qbar}
\#\bigl(\rH^1_\qq(K_{\qqbar},T)/\rH^1(K_{\qqbar},T)^u\bigr)
=\#\bigl(\rH^1(K_{\qqbar},T)/\rH^1(K_{\qqbar},T)^u\bigr)=\#E(\Q_p)[p^\infty].
\end{equation}
It remains to consider the case $w\nmid p$. In this case, we have
\[
\rH^1_\qq(K_w,T)=\rH^1(K_w,T)_{\rm tor}=\rH^1(K_w,T),
\]
using Tate's local Euler characteristic formula for the second equality. Therefore,
\begin{equation}\label{eq:tam-w}
\begin{aligned}
\#\bigl(\rH^1_\qq(K_w,T)/\rH^1(K_w,T)^u\bigr)&=
\#\bigl(\rH^1(K_w,T)/\rH^1(K_w,T)^u\bigr)\\
&=\#\bigl(\varinjlim_k\rH^0(K_w,S^{(k)})_\Gamma\bigr)=\#\bigl(B_w/(\gamma-1)B_w\bigr),
\end{aligned}
\end{equation}
where the second equality is shown in the proof of Proposition~\ref{prop:self-dual}, and $B_w$ is as in \eqref{eq:def-B}. Now recall that $c_w^{(p)}=\#\rH^1_{\rm unr}(K_w,W)$, where
\[  
\rH^1_{\rm unr}(K_w,W):={\rm ker}\left\{\rH^1(K_w,W)\rightarrow\rH^1(K_w^{\rm unr},W)\right\}.
\]
Since for every $\eta\mid w$ the restriction map $\rH^1(K_{\infty,\eta},W)\rightarrow\rH^1(K_w^{\rm unr},W)$ is injective (this follows from the fact that ${\rm Gal}(K_w^{\rm unr}/K_{\infty,\eta})$ has trivial pro-$p$-part), we deduce that
\begin{equation}\label{eq:tam-coinv}
\begin{aligned}
\rH^1_{\rm unr}(K_w,W)&={\rm ker}\{\rH^1(K_w,W)\rightarrow\bigoplus_{\eta\mid w}\rH^1(K_{\infty,\eta},W)\}\\
&
\simeq B_w/(\gamma-1)B_w.
\end{aligned}
\end{equation}
From \eqref{eq:tam-w} and \eqref{eq:tam-coinv}, we conclude that $\#\bigl(\rH^1_\qq(K_w,T)/\rH^1(K_w,T)^u\bigr)=c_w^{(p)}$ for $w\nmid p$.  Together with \eqref{eq:order-q} and \eqref{eq:order-qbar}, this gives the stated formula for $\#\mathcal{L}$.
\end{proof}



\subsection{Step 4: $p$-adic logarithm} 

The last step is to relate $\#\bigl({\rm Sel}_{\ppbar}(K,W)_{/{\rm div}}\bigr)$ to $\#\Sha_{\rm BK}(K,W)$ and some of the other terms appearing in our Leading Coefficient Formula.

\begin{prop}\label{prop:sel-sha}
Assume hypotheses {\rm \eqref{eq:h0-intro}}-{\rm \eqref{eq:h1-intro}}. 
Then
\[
\#({\rm Sel}_{\ppbar}(K,W)_{/{\rm div}})=\#\Sha_{\rm BK}(K,W)\cdot(\#{\rm coker}({\rm res}_{\pp/{\rm tor}}))^2,
\]
where ${\rm res}_{\pp/{\rm tor}}$ is the composition
$\check{S}_p(E/K)\xrightarrow{{\rm res}_\pp}E(K_\pp)\otimes\Z_p\rightarrow E(K_\pp)_{/{\rm tor}}\otimes\Z_p$. 
\end{prop}

\begin{proof}
By Lemma~\ref{lem:rank-1} we have 
\[
{\rm Sel}_\pp(K,T)=\ker({\rm res}_\pp)\simeq\Z_p^{r-1},
\] 
where $r={\rm rank}_{\Z_p}\check{S}_p(E/K)$.
Let $s_1,\dots,s_{r-1}$ be a $\Z_p$-basis for ${\rm Sel}_\pp(K,T)$ and extend it to a $\Z_p$-basis $s_1,\dots,s_{r-1},s_\pp$ for $\check{S}_p(E/K)$, so we have
\begin{equation}\label{eq:free}
\check{S}_p(E/K)\simeq{\rm Sel}_\pp(K,T)\oplus\Z_ps_\pp.
\end{equation}

The map ${\rm res}_{\pp/{\rm tor}}$ gives an injection $\Z_ps_\pp\hookrightarrow E(K_\pp)_{/{\rm tor}}\otimes\Z_p$, and defining $U$ by the exactness of the sequence
\begin{equation}\label{eq:U}
0\rightarrow\Z_ps_\pp\rightarrow E(K_\pp)_{/{\rm tor}}\otimes\Z_p\rightarrow U\rightarrow 0,
\end{equation}
we see that $U$ is finite, with $\#U=\#{\rm coker}({\rm res}_{\pp/{\rm tor}})$.
Tensoring \eqref{eq:U} with $\Q_p/\Z_p$ gives
\begin{equation}\label{eq:V}
0\rightarrow U'\rightarrow(\Q_p/\Z_p)s_\pp\rightarrow E(K_\pp)\otimes\Q_p/\Z_p\rightarrow 0
\end{equation}
for a certain finite module $U'$ with $\#U'=\#U$.

Noting that ${\rm Sel}_\pp(K,T)\otimes_{\Z_p}\Q_p/\Z_p={\rm Sel}_\pp(K,W)_{\rm div}$ (see \eqref{eq:div}) 
and similarly 
$\check{S}_p(E/K)\otimes_{\Z_p}\Q_p/\Z_p={\rm Sel}_{p^\infty}(E/K)_{\rm div}$, from \eqref{eq:free} and \eqref{eq:V} it follows that
\begin{equation}\label{eq:I}
\ker\biggl\{{\rm Sel}_{p^\infty}(E/K)_{\rm div}\xrightarrow{{\rm res}_\pp}E(K_\pp)\otimes\Q_p/\Z_p\biggr\}\simeq{\rm Sel}_\pp(K,W)_{\rm div}\oplus U'.
\end{equation}
Moreover, we know that 
\begin{equation}\label{eq:order-V}
\#U'=\#{\rm coker}({\rm res}_{\pp/{\rm tor}}).
\end{equation}

Denote by ${\rm Sel}_{{\rm rel},{\rm fin}}(K,V)$ the Selmer group obtained by relaxing the condition at $\pp$ in the usual Bloch--Kato Selmer group ${\rm Sel}(K,V)$, and let ${\rm Sel}_{{\rm rel},{\rm fin}}(K,W)$ be the Selmer group obtained by propagation.
Then from Poitou--Tate duality we have the exact sequence
\begin{align*}
0\rightarrow{\rm Sel}_{p^\infty}(E/K)\rightarrow{\rm Sel}_{{\rm rel},{\rm fin}}(K,W)\xrightarrow{{\rm res}_{\pp/{\rm fin}}}\frac{\rH^1(K_{\pp},W)_{\rm div}}{E(K_{\pp})\otimes\Q_p/\Z_p}&\simeq(E(K_{\pp})_{/{\rm tor}}\otimes\Z_p)^\vee\\
&\quad\xrightarrow{({\rm res}_{\pp/{\rm tor}})^\vee}\check S_p(E/K)^\vee,
\end{align*}
and it follows from our assumption that the map ${\rm res}_{\pp/{\rm tor}}$ has finite cokernel. Hence the map ${\rm res}_{{\pp}/{\rm fin}}$ has finite image, with
\begin{equation}\label{eq:duality}
\#{\rm im}({\rm res}_{\pp/{\rm fin}})=\#{\rm coker}({\rm res}_{\pp/{\rm tor}}).
\end{equation}
We thus deduce the following commutative diagram with exact rows
\begin{equation}\label{eq:diag}
\begin{aligned}
\xymatrix{
&{\rm Sel}_{p^\infty}(E/K)_{\rm div}\ar[r]^-{\simeq}\ar@{^{(}->}[d]&{\rm Sel}_{{\rm rel},{\rm fin}}(K,W)_{\rm div}\ar@{^{(}->}[d]&\\
0\ar[r]&{\rm Sel}_{p^\infty}(E/K)\ar[r]&{\rm Sel}_{{\rm rel},{\rm fin}}(K,W)\ar[r]&{\rm im}({\rm res}_{\pp/{\rm fin}})\ar[r]&0,
}
\end{aligned}
\end{equation}
from where, together with \eqref{eq:duality}, we conclude that
\begin{equation}\label{eq:step1}
\#({\rm Sel}_{{\rm rel},{\rm fin}}(K,W)_{/{\rm div}})=\#\Sha_{\rm BK}(K,W)\cdot\#{\rm coker}({\rm res}_{\pp/{\rm tor}}).
\end{equation}

On the other hand, from \eqref{eq:I} and the isomorphism in \eqref{eq:diag} we also have the commutative diagram with exact rows 
\[
\xymatrix{
0\ar[r]&{\rm Sel}_{\ppbar}(K,W)_{\rm div}\oplus U'\ar[r]\ar@{^{(}->}[d]&{\rm Sel}_{{\rm rel},{\rm fin}}(K,W)_{\rm div}\ar@{^{(}->}[d]\ar[r]^-{{\rm res}_{\ppbar}}&E(K_{\ppbar})\otimes\Q_p/\Z_p\ar@{=}[d]\ar[r]&0\\
0\ar[r]&{\rm Sel}_{\ppbar}(K,W)\ar[r]&{\rm Sel}_{{\rm rel},{\rm fin}}(K,W)\ar[r]^-{{\rm res}_{\ppbar}}&E(K_{\ppbar})\otimes\Q_p/\Z_p\ar[r]&0,
}
\]
from where we conclude that
\begin{align*}
\#({\rm Sel}_{\ppbar}(K,W)_{/{\rm div}})&=\#({\rm Sel}_{{\rm rel},{\rm fin}}(K,W)_{/{\rm div}})\cdot\#U'.
\end{align*}
Together with \eqref{eq:order-V} and \eqref{eq:step1}, this last formula yields the result.
\end{proof}

Recall that we use ${\rm log}_\pp:\check{S}_p(E/K)\rightarrow\Z_p$ to denote the composition of ${\rm res}_{\pp/{\rm tor}}:\check{S}_p(E/K)\rightarrow E(K_\pp)_{/{\rm tor}}\otimes\Z_p$ with the formal group logarithm 
associated with a fixed N\'{e}ron differential $\omega_E\in\Omega^1(E/\Z_{(p)})$.

\begin{prop}\label{prop:coker}
Let the hypotheses be as in Proposition\;\ref{prop:sel-sha}.
Then
\[
\#{\rm coker}({\rm res}_{\pp/{\rm tor}})\sim_p
\biggl(\frac{1-a_p(E)+p}{p}\biggr)\cdot\log_\pp(s_\pp)\cdot\frac{1}{\#E(K_\pp)[p^\infty]},
\]
where $s_\pp$ is any generator of $\check{S}_p(E/K)/\ker({\rm res}_{\pp/{\rm tor}})\simeq\Z_p$ and $a_p(E):=p+1-\#E(\mathbf{F}_p)$.
\end{prop}

\begin{proof}
As shown in the proof of Proposition\,\ref{prop:sel-sha} (see \eqref{eq:U}), we have
\begin{equation}\label{eq:index}
\#({\rm coker}({\rm res}_{\pp/{\rm tor}}))=[E(K_\pp)_{/{\rm tor}}\otimes\Z_p:\Z_ps_\pp].
\end{equation}
Let $E_1(K_\pp)$ be the kernel of the reduction map modulo $\pp$, so 
we have the exact sequence
\begin{equation}\label{eq:red-p}
0\rightarrow E_1(K_\pp)\rightarrow E(K_\pp)\rightarrow E(\mathbf{F}_p)\rightarrow 0.\nonumber
\end{equation}
The formal group logarithm defines an isomorphism
\[
{\rm log}_\pp:E_1(K_\pp)\otimes\Z_p\xrightarrow{\simeq}p\Z_p,
\]
and this extends to an injection $E(K_\pp)_{/{\rm tor}}\otimes\Z_p\hookrightarrow\Z_p$.
Hence from \eqref{eq:index} we find
\begin{align*}
\#({\rm coker}({\rm res}_{\pp/{\rm tor}}))&=\frac{[\Z_p:{\rm log}_\pp(s_\pp)\Z_p]}{[\Z_p:{\rm log}_\pp(E(K_\pp)_{/{\rm tor}}\otimes\Z_p)]}\\
&=\frac{[\Z_p:{\rm log}_\pp(s_\pp)\Z_p]}{[\Z_p:p\Z_p]}\cdot[E(K_\pp)_{/{\rm tor}}\otimes\Z_p:E_1(K_\pp)\otimes\Z_p]\\
&\sim_p\frac{\log_\pp(s_\pp)}{p}\cdot\frac{[E(K_\pp)\otimes\Z_p:E_1(K_\pp)\otimes\Z_p]}{\#E(K_\pp)[p^\infty]}.
\end{align*}
Since by definition $\#(E(\mathbf{F}_p)\otimes\Z_p)\simeq\#\Z_p/(1-a_p(E)+p)\Z_p$, this yields the result.
\end{proof}


\subsection{Leading Coefficient Formula}


\begin{proof}[Proof of Theorem~\ref{thm:A}(i)]
Let $r\geq\varrho_{\rm alg}={\rm ord}_JF_{\ppbar}^{\rm BDP}$.
From Proposition~\ref{prop:well-known} and Proposition~\ref{prop:recursion} we have the equalities up to a $p$-adic unit:
\begin{equation}\label{eq:step1-2}
\begin{aligned}
\frac{1}{\varrho_{\rm alg}!}\frac{d^{\varrho_{\rm alg}}}{dT^{\varrho_{\rm alg}}}F_{\ppbar}^{\rm BDP}\Bigl\vert_{T=0}\,&\sim_p\,\#\bigl(\underrightarrow{\mathfrak{S}}_{\ppbar}^{(r+1)}\bigr)\\
\,&\sim_p\,{\rm Reg}_{\pp,{\rm der},\gamma}\cdot\bigl[{\rm Sel}_\pp(K,T)\colon\mathfrak{Sel}_\pp(K,T)\bigr]\cdot
\#\bigl(\mathfrak{Sel}_{\ppbar}(K,W)_{/{\rm div}}\bigr).
\end{aligned}
\end{equation}
Since from Corollary~\ref{cor:control-2} we have the relation
\[
\bigl[{\rm Sel}_\pp(K,T)\colon\mathfrak{Sel}_{\pp}(K,T)\bigr]\cdot
\#\bigl(\mathfrak{Sel}_{\ppbar}(K,W)_{/{\rm div}}\bigr)
=\#(E(\Q_p)[p^\infty])^2\cdot\prod_{w\mid N}c_w^{(p)}\cdot\#({\rm Sel}_{\ppbar}(K,W)_{/{\rm div}}),
\]
continuing from \eqref{eq:step1-2} we deduce that
\begin{align*}
\frac{1}{\varrho_{\rm alg}!}\frac{d^{\varrho_{\rm alg}}}{dT^{\varrho_{\rm alg}}}&F_{\ppbar}^{\rm BDP}\Bigl\vert_{T=0}\\
\,&\sim_p\,{\rm Reg}_{\pp,{\rm der},\gamma}\cdot(\#E(\Q_p)[p^\infty])^2\cdot\prod_{w\mid N}c_w^{(p)}\cdot\#({\rm Sel}_{\ppbar}(K,W)_{/{\rm div}})\\
\,&\sim_p\,{\rm Reg}_{\pp,{\rm der},\gamma}\cdot(\#E(\Q_p)[p^\infty])^2\cdot\prod_{w\mid N}c_w^{(p)}\cdot\#\Sha_{\rm BK}(K,W)\cdot(\#{\rm coker}({\rm res}_{\pp/{\rm tor}}))^2\\
\,&\sim_p\,{\rm Reg}_{\pp,{\rm der},\gamma}\cdot\biggl(\frac{1-a_p(E)+p}{p}\biggr)^2\cdot{\rm log}_\pp(s_\pp)^2\cdot\prod_{w\mid N}c_w^{(p)}\cdot\#\Sha_{\rm BK}(K,W),
\end{align*}
using Proposition~\ref{prop:sel-sha}
 and Proposition~\ref{prop:coker} for the middle and the last equality, respectively.
 Noting that $\prod_{w\mid N}c_w=\prod_{\ell\mid N}c_\ell^2$ as a consequence of \eqref{eq:Heeg}, this finishes the proof.
\end{proof}

\subsection{Order of Vanishing}\label{subsec:order-van}

In this section we give the proof of Theorem~\ref{thm:A}(ii).
Since $\underleftarrow{\overline{\mathfrak{S}}}_\qq^{(1)}={\rm Sel}_\qq(K,T)$, from \eqref{eq:ht-Qp} we have a $p$-adic height pairing
\[
h_\pp = h_\pp^{(1)}:{\rm Sel}_\pp(K,T)\times{\rm Sel}_{\ppbar}(K,T)\rightarrow\Q_p
\]
whose kernel on the left 
is given by $\underleftarrow{\overline{\mathfrak{S}}}_\pp^{(2)}$ (and whose $\Z_p$-rank is the same as that of $\underleftarrow{{\mathfrak{S}}}_\pp^{(2)}$). 

\begin{prop}\label{prop:e2}
Set $r^\pm={\rm rank}_{\Z_p}\check{S}_p(E/K)^\pm$.
Then
\[
{\rm rank}_{\Z_p}\underleftarrow{\mathfrak{S}}_\pp^{(2)}\geq\abs{ r^+-r^-}-1.
\]
\end{prop}

\begin{proof}
Note that complex conjugation 
acts on ${\rm Sel}_{\rm str}(K,T)={\rm Sel}_\pp(K,T)\cap{\rm Sel}_{\ppbar}(K,T)$.
Let 
\begin{equation}
\label{eq:h-str}
h_{\rm str}:{\rm Sel}_{\rm str}(K,T)\times{\rm Sel}_{\rm str}(K,T)\rightarrow\Q_p\noindent
\end{equation}
be the pairing obtained from $h_\pp$ by restriction.
By \cite[Rem.~1.12]{howard-derived}, we have
\[
h_{\rm str}(x^\tau,y^\tau)=-\check{h}_p(x,y),
\]
for all $x,y\in{\rm Sel}_{\rm str}(K,T)$, where $\tau$ is complex conjugation.
Writing $r_{\rm str}^\pm$ to denote the $\Z_p$-rank of the $\tau$-eigenspace ${\rm Sel}_{\rm str}(K,T)^{\pm}$, it follows that
\begin{equation}\label{eq:null-str}
{\rm rank}_{\Z_p}{\rm ker}(h_{\rm str})\geq\abs{ r_{\rm str}^+-r_{\rm str}^-}.
\end{equation}
We distinguish two cases according to the $\Z_p$-rank of the image of the restriction map ${\rm res}_p$ in the proof of Lemma~\ref{lem:rank-1}.
\sk

\noindent\emph{Case (i): ${\rm rank}_{\Z_p}{\rm im}({\rm res}_p)=1$.} 
By \eqref{eq:BDP=str}, we have $h_{\rm str}=h_\pp$ and
\begin{equation}\label{eq:drops}
(r_{\rm str}^+,r_{\rm str}^-)\in\{(r^+-1,r^-),(r^+,r^--1)\}.\nonumber
\end{equation}
Thus $\abs{ r_{\rm str}^+-r_{\rm str}^-}\geq\abs{ r^+-r^-}-1$, and the result follows from \eqref{eq:null-str}.
\sk

\noindent\emph{Case (ii): ${\rm rank}_{\Z_p}{\rm im}({\rm res}_p)=2$.} 
Consider a non-zero element $z\in\check{S}_p(E/K)$ satisfying ${\rm res}_\pp(z)={\rm res}_{\ppbar}(z^\tau)=0$ and ${\rm res}_{\ppbar/{\rm tor}}(z)\neq 0$.
Then 
\begin{equation}\label{eq:S-Sstr}
\check{S}_p(E/K)={\rm Sel}_{\rm str}(K,T)\oplus\Z_pz^+\oplus\Z_pz^-,
\end{equation}
where $z^\pm=\frac{1}{2}(z\pm z^\tau)$.
With this notation, we can write
\begin{align*}
%
{\rm Sel}_\pp(K,T)={\rm Sel}_{\rm str}(K,T)\oplus\Z_pz,
\qquad{\rm Sel}_{\ppbar}(K,T)={\rm Sel}_{\rm str}(K,T)\oplus\Z_pz^\tau.
\end{align*}
Now, we immediately see that
\begin{align*}
\textrm{(left kernel of $h_\pp$)}\supset\bigcap_{s}{\rm ker}(h_{\rm str}(-,s))\cap{\rm ker}(h_\pp(-,z^\tau)),
\end{align*}
where $s$ runs over all the elements in ${\rm Sel}_{\rm str}(K,T)$.
Thus we conclude that
\begin{align*}
{\rm rank}_{\Z_p}\underleftarrow{\mathfrak{S}}_\pp^{(2)}&\geq{\rm rank}_{\Z_p}{\rm ker}(h_{\rm str})-1\\
&\geq\abs{ r_{\rm str}^+-r_{\rm str}^-}-1\\
&=\abs{ r^+-r^-}-1,
\end{align*}
using \eqref{eq:null-str} and \eqref{eq:S-Sstr} for the second inequality and the last equality, respectively.
\end{proof}

\begin{rem}\label{rem:sha}
Conjecturally, \emph{Case (i)} in the proof of Proposition~\ref{prop:e2} only occurs when either $r^+$ or $r^-$ is $0$.
Indeed, let $E^K/\Q$ be the twist of $E$ by the quadratic character corresponding to $K/\Q$.
If both $r^+$ and $r^-$ are positive, then the finiteness of $\Sha(E/K)[p^\infty]=\Sha(E/\Q)[p^\infty]\oplus\Sha(E^K/\Q)[p^\infty]$ implies that the restriction map
\[
{\rm res}_p=({\rm res}_p^+,{\rm res}_p^-):\check{S}_p(E/K)\rightarrow E(K_p)\otimes\Z_p=(E(\Q_p)\otimes\Z_p)\oplus(E^K(\Q_p)\otimes\Z_p)
\]
satisfies ${\rm rank}_{\Z_p}{\rm im}({\rm res}_p^\pm)=1$, so the $\Z_p$-rank of ${\rm im}({\rm res}_p)$ is $2$.
\end{rem}

The following is Theorem~\ref{thm:A}(ii):

\begin{cor}\label{cor:rank-ineq}
Let ${\varrho_{\rm alg}}={\rm ord}_{J}F_{\ppbar}^{\rm BDP}$.
Then
\[
{\varrho_{\rm alg}}\geq 2(\max\{r^+,r^-\}-1),
\]
where $r^\pm={\rm rank}_{\Z_p}\check{S}_p(E/K)^\pm$, with equality if and only if $h_\pp^{}$ is maximally non-degenerate.
\end{cor}

\begin{proof}
With the notation from \eqref{eq:varrho}, we have
\begin{align*}
{\varrho_{\rm alg}}
&=e_1+2e_2+\cdots+i_0e_{i_0}\\
&\geq (e_1+e_2+\cdots+e_{i_0})+(e_2+\cdots+e_{i_0})\\
&\geq(r-1)+(\abs{ r^+-r^-}-1)\\
&=2(\max\{r^+,r^-\}-1),
\end{align*} 
using Proposition~\ref{prop:e2} for the second inequality.
These inequalities are equalities if and only if $e_i=0$ for $i\geq 3$ and $e_2=\abs{ r^+-r^-}-1$, as was to be shown.
\end{proof}

\begin{rem}\label{rem:max-nondeg}
Assume that $p$ is ordinary for $E$.
Applied to the usual (compact) Selmer group $\check{S}_p(E/K)$, Howard's work produces a filtration
\[
\check{S}_p(E/K)\otimes\Q_p=S_p^{(1)}\supset S_p^{(2)}\supset\cdots \supset S_p^{(i)}\supset\cdots
\]
and a sequence of pairings $\check{h}_p^{(i)}:S_p^{(i)}\times S_p^{(i)}\rightarrow\Q_p$ such that $S_p^{(i+1)}$ is the kernel of $S_p^{(i)}$ (\cite[Thm.~4.2]{howard-derived}).
In this setting, a conjecture due to Mazur and Bertolini--Darmon predicts that
\begin{equation}\label{eq:BD-conj}
{\rm dim}_{\Q_p}S_p^{(i)}\overset{?}=\begin{cases}
\vert r^+-r^-\vert&\textrm{if $i=2$,}\\
1&\textrm{if $i=3$,}\\[0.2em]
\end{cases}
\end{equation}
and $S_p^{(i)}=0$ for $i\geq 4$, and that $S_p^{(3)}$ spanned by the space of universal norms 
\[
\check{S}_p(E/K)^u=\bigcap_n{\rm cor}_{K_n/K}(\check{S}_p(E/K_n)) 
\]
(see \cite[Conj.~4.4]{howard-derived} and \cite[Conj.~3.8]{BD-derived-AJM}).
Using Proposition~\ref{prop:Lambda-tors} and Corollary~\ref{cor:der-ht}, one easily checks 
(arguing similarly as in the proof of Proposition~\ref{prop:e2}) that $\check{S}_p(E/K)^u\cap{\rm Sel}_\pp(K,T)=0$ .
As a result, conjecture \eqref{eq:BD-conj} implies that $h_\pp$ is maximally non-degenerate.

In the $p$-supersingular case the same conclusion should hold, building on the work of Benois \cite{benois-ht} to obtain (derived) anticyclotomic $p$-adic height pairings on $\check{S}_p(E/K)\otimes\Q_p$ compatible with our $h_\pp$.  
\end{rem}


\section{Proof of anticyclotomic main conjectures: supersingular case}\label{sec:IMC}

The purpose of this section is to give a proof of the signed Heegner point main conjecture formulated in \cite{CW-PR} for supersingular primes $p$ under mild hypotheses. 
By the equivalence between this conjecture and Conjecture~\ref{conj:BDP-IMC} when $p=\pp\ppbar$ splits in $K$, we deduce a proof of the latter conjecture under the same hypotheses.

The formulation of the signed main conjecture in \cite{CW-PR} is under a generalised Heegner hypothesis\footnote{Allowing $N$ to be divisible by any square-free product $N^-$ of an even number of primes inert in $K$.}, 
and many cases were proved in \emph{op.\,cit.} by building on the main result of \cite{CLW}.
Unfortunately, due to the technical hypotheses from \cite{CLW}, the classical Heegner hypothesis \eqref{eq:Heeg}, i.e. the case  $N^-=1$,   
was excluded from those results.
To obtain a result under \eqref{eq:Heeg}, here we adapt the approach of \cite{BCK} to the supersingular setting\footnote{After a first draft of this Appendix was written, similar results were announced in \cite{BLV}.}.



\subsection{Statement of the main results}\label{subsec:statement}

We begin by introducing the setting, which is slightly more general than what is needed for the application in Corollary~\ref{cor:B}.

Let $E/\Q$ be an elliptic curve of conductor $N$, and let $p>2$ be a prime of good supersingular reduction for $E$.
We assume that
\begin{equation}\label{eq:non-ord}
a_p(E)=0\tag{non-ord}
\end{equation} 
(which by the Hasse bounds is automatic for $p>3$). Let $K$ be an imaginary quadratic field of discriminant prime to $N$ such that
\begin{equation}\label{eq:split-app}
\textrm{$p=\pp\ppbar$ splits in $K$.}\tag{spl}
\end{equation}
Writing $N=N^+N^-$ with $N^+$ (resp. $N^-$) divisible only by primes that split (resp. remain inert) in $K$, assume the following \emph{generalised Heegner hypothesis}:
\begin{equation}\label{eq:gen-Heeg}
\textrm{$N^-$ is the square-free product of an even number of primes.}\tag{gen-Heeg}
\end{equation}
As in the main text, we let $\Gamma$ be the Galois group of the anticyclotomic $\Z_p$-extension $K_\infty/K$, and let $\Lambda=\Z_p\dBr{\Gamma}$ be the anticyclotomic Iwasawa algebra.
%

For any triple $(E,K,p)$ satisfying \eqref{eq:non-ord}, \eqref{eq:split-app}, and \eqref{eq:gen-Heeg} as above, an analogue of Perrin-Riou's Heegner point main conjecture 
\cite{PR-HP} was formulated in \cite[Conj.~4.8]{CW-PR}.
Letting $\mathcal{S}^\pm$ and $\mathcal{X}^\pm$ be the compact Selmer groups over $K_\infty/K$ denoted by ${\rm Sel}^\pm(K,\mathbf{T}^{\rm ac})$ and ${\rm Sel}^\pm(K,\mathbf{A}^{\rm ac})^\vee$ in \cite[\S{4.2}]{CW-PR}, respectively, the conjecture predicts that both $\mathcal{S}^\pm$ and $\mathcal{X}^\pm$ have $\Lambda$-rank one, 
with the characteristic ideal of the $\Lambda$-torsion submodule $\mathcal{X}_{\rm tors}^\pm\subset\mathcal{X}^\pm$ being the same as ${\rm char}_\Lambda(\mathcal{S}^\pm/(\kappa_\infty^\pm))^2$ for the $\Lambda$-adic $\pm$-Heegner class
\begin{equation}\label{eq:pm-Heeg}
\kappa_\infty^\pm\in\mathcal{S}^\pm
\end{equation}  
constructed in \cite[$\S{4.1}$]{CW-PR} (where it is denoted $\mathbf{z}_\infty^\pm={\rm cor}_{K[1]/K}(\mathbf{z}_\infty[1]^\pm)$.

In this section we prove the following.



\begin{thm}\label{thm:pm-HPMC}
Let $(E,K,p)$ be a triple as above, and assume in addition that:
\begin{itemize}
\item[(i)] $N$ is square-free.
\item[(ii)] $E[p]$ is ramified at every prime $\ell\mid N^+$. 
\item[(iii)] $E[p]$ is ramified at every prime $\ell\mid N^-$ with $\ell\equiv\pm{1}\pmod{p}$.
\item[(iv)] Every prime above $p$ is totally ramified in $K_\infty/K$.
\end{itemize}
Then $\mathcal{S}^\pm$ and $\mathcal{X}^\pm$ have $\Lambda$-rank one, and 
\[
{\rm char}_\Lambda(\mathcal{X}^\pm_{\rm tors})={\rm char}_\Lambda(\mathcal{S}^\pm/(\kappa_\infty^\pm))^2.
\] 
In other words, the $\pm$-Heegner point main conjecture in \cite[Conj.~4.8]{CW-PR} holds.
\end{thm}


\begin{cor}\label{cor:BDP-IMC}
Let the hypotheses be as in Theorem~\ref{thm:pm-HPMC}.
Then \cite[Conj.~5.2]{CW-PR} holds.
In particular, in the case $N^-=1$,  Conjecture~\ref{conj:BDP-IMC} in the body of the paper holds.
\end{cor}

\begin{proof}
This follows from Theorem~\ref{thm:pm-HPMC} and the equivalence in \cite[Thm.~6.8]{CW-PR}, noting that Conjecture~\ref{conj:BDP-IMC} is the same as \cite[Conj.~5.2]{CW-PR} when $N^-=1$. 
\end{proof}


The remainder of this section is devoted to the proof of Theorem~\ref{thm:pm-HPMC}.

\subsection{Bipartite Euler system for non-ordinary primes}

Let $f=\sum_{n=1}^\infty a_nq^n\in S_2(\Gamma_0(N))$ be the newform corresponding to $E$, and denote by
\[
\overline{\rho}:G_\Q\rightarrow{\rm Aut}_{\mathbf{F}_p}(E[p])\simeq{\rm GL}_2(\mathbf{F}_p)
\]
the associated residual representation.
%
%
Following \cite{pollack-weston}, we say that the pair $(\overline{\rho},N^-)$ satisfies \emph{Condition CR} if:
\begin{itemize}
\item $\overline{\rho}$ is ramified at every prime $\ell\mid N^-$ with $\ell\equiv\pm{1}\pmod{p}$, and
\item $\overline{\rho}$ is surjective\footnote{It follows from \cite[Prop.~2.1]{edix} that if $E$ is semistable and $p$ is supersingular for $E$, then $\overline{\rho}$ is surjective.}.
\end{itemize}


We refer the reader to 
\cite[p.\,18]{bdIMC} for the definition of \emph{$j$-admissible} primes (for any $j>0$) relative to $f$. 
Denote by $\mathcal{L}_j$ the set of $j$-admissible primes, and by $\mathcal{N}_j$ the set of square-free products of primes $q\in\mathcal{L}_j$.
When $j=1$, we suppress it from the notations.
We decompose
\[
\mathcal{N}_j=\mathcal{N}_j^{\rm ind}\sqcup\mathcal{N}_j^{\rm def}
\]
with $\mathcal{N}_j^{\rm ind}$ 
consisting of the square-free products of an even 
number of primes $q\in\mathcal{L}_j$.

\begin{rem}
By definition, admissible primes $q\in\mathcal{L}$ satisfy in particular $q\not\equiv\pm{1}\pmod{p}$.
So, Condition~CR allows for the existence, for any $m\in\mathcal{N}$, of $m$-new forms $g\in S_2(\Gamma_0(Nm))$ level-raising $f$ (and whose existence follows from results of Ribet \cite{ribet-ICM} and Diamond--Taylor \cite{DT-dmj,DT-inv}; see \cite[Thm.~2.1]{zhang-Kolyvagin}).
\end{rem}

Let $T=\varprojlim_jE[p^j]$ be the $p$-adic Tate module of $E$, and put
\[
\mathbf{T}_j:=\varprojlim_n{\rm Ind}_{K_n/K}(T/p^jT),\qquad
\mathbf{A}_j:=\varinjlim_n{\rm Ind}_{K_n/K}(E[p^j]).
\]
For every $m\in\mathcal{N}_j$ the ``$N^-m$-ordinary'' signed Selmer groups ${\rm Sel}_{N^-m}^\pm(K,\mathbf{T}_j)$, ${\rm Sel}_{N^-m}^\pm(K,\mathbf{A}_j)$ are defined as in \cite[p.\,1634]{BCK}, with the local conditions at primes $v\mid p$ in \emph{loc.\,cit.} replaced by the above local conditions $\rH^1_\pm(K_v,\mathbf{T}_j)$ and $\rH^1_\pm(K_v,\mathbf{A}_j)$  in \cite[Def.~4.6]{CW-PR}.

In particular, at the primes $q\nmid N^-mp$, the classes $c\in{\rm Sel}_{N^-m}^\pm(K,\mathbf{T}_j)$ are unramified, i.e., ${\rm res}_q(c)\in\rH^1_{\rm unr}(K_q,\mathbf{T}_j)$, 
while at the primes $q\mid N^-m$ they are required to land in the ``ordinary'' submodule $\rH^1_{\rm ord}(K_q,\mathbf{T}_j)$. 
It is easy to see that for $q\in\mathcal{L}$, both $\rH^1_{\rm unr}(K_q,\mathbf{T}_j)$ and $\rH^1_{\rm ord}(K_q,\mathbf{T}_j)$ are free of rank $1$ over $\Lambda/p^j\Lambda$ (see e.g. \cite[Lem.~2.1]{BCK}).



\begin{thm}[Darmon--Iovita, Pollack--Weston]\label{thm:construction}
Suppose that 
\begin{itemize}
\item[(i)] $a_p(E)=0$.
\item[(ii)] $p$ splits in $K$.
\item[(iii)] Each prime above $p$ is totally ramified in $K_\infty/K$.
\item[(iv)] $(\overline{\rho},N^-)$ satisfies Condition~CR.
\end{itemize}
Then 
for every choice of sign $\pm$ and every $j>0$ there is a pair of systems
\begin{align*}
\boldsymbol{\kappa}^\pm_j&=\{\kappa^\pm_j(m)\in{\rm Sel}^\pm_{N^-m}(K,\mathbf{T}_j)\,\colon\,m\in\mathcal{N}_j^{\rm ind}\},\\
\boldsymbol{\lambda}_j^\pm&=\{\lambda_j^\pm(m)\in\Lambda/\wp^j\Lambda\,\colon\,m\in\mathcal{N}_j^{\rm def}\},
\end{align*}
related by a system of ``explicit reciprocity laws'':
\begin{itemize}
\item If $mq_1q_2\in\mathcal{N}_j^{\rm ind}$ with $q_1,q_2\in\mathcal{L}_j$ distinct primes, then
\[
{\rm loc}_{q_2}(\kappa_j^\pm(mq_1q_2))=\lambda^\pm_j(mq_1)
\]
under a fixed isomorphism $\rH^1_{\rm ord}(K_{q_1},\mathbf{T}_j)\simeq\Lambda/p^j\Lambda$.
\item If $mq\in\mathcal{N}_j^{\rm def}$ with $q\in\mathcal{L}_j$ prime, then
\[
{\rm loc}_q(\kappa_j^\pm(m))=\lambda_j^\pm(mq)
\]
under a fixed isomorphism $\rH^1_{\rm unr}(K_q,\mathbf{T}_j)\simeq\Lambda/p^j\Lambda$.
\end{itemize}
\end{thm}

\begin{proof}
This is shown in \cite{darmon-iovita} (in particular, see [\emph{op.\,cit.}, Prop.~4.4, Prop.~4.6] for the two explicit reciprocity laws) under hypotheses (i)--(iii) and an additional hypothesis that 
$f$ is ``$p$-isolated'' in the sense of \cite{bdIMC}.
This last hypothesis was replaced by the weaker hypothesis (iv) above in \cite{pollack-weston} (see  [\emph{op.\,cit.}, $\S4.3$]).
\end{proof}

For $m=1$, the classes $\kappa_j^\pm:=\kappa_j^\pm(1)$ exist for all $j>0$ and are compatible under the natural maps ${\rm Sel}_{N^-}^\pm(K,\mathbf{T}_{j+1})\rightarrow{\rm Sel}_{N^-}^\pm(K,\mathbf{T}_j)$, thereby defining the class
\begin{equation}\label{eq:spec-class}
\varprojlim_j\kappa_j^\pm\in{\rm Sel}_{N^-}^\pm(K,\mathbf{T}):=\varprojlim_j{\rm Sel}_{N^-}^\pm(K,\mathbf{T}_j).
\end{equation}

As in \cite[Lem.~2.2]{BCK}, we have natural isomorphisms
\[
\mathcal{S}^\pm\simeq\varprojlim_j{\rm Sel}_{N^-}^\pm(K,\mathbf{T}_j),\qquad
\mathcal{X}^\pm\simeq\biggl(\varinjlim_j{\rm Sel}_{N^-}^\pm(K,\mathbf{A}_j)\biggr)^\vee
\]
(note that the almost divisibility result of \cite[Prop.~3.12]{hatley-lei-MRL} used in the proof can be shown in the same manner in the supersingular case, replacing the appeal to results from \cite{cas-BF} by their counterparts in \cite{CW-PR}, and using the equality 
${\rm Sel}^\pm(K,\mathbf{T}^{\rm ac})={\rm Sel}^{\pm,{\rm rel}}(K,\mathbf{T}^{\rm ac})$ shown in the proof of \cite[Thm.~6.8]{CW-PR}).  Moreover, comparing the construction of the classes $\kappa_j^\pm(m)$ in \cite[$\S{4}$]{darmon-iovita}\footnote{see esp. \cite[Prop.~4.3]{darmon-iovita}.} and the construction of the classes $\mathbf{z}_\infty[S]^\pm$ in \cite[$\S{4.1}$]{CW-PR}\footnote{see esp. \cite[Prop.~4.4]{CW-PR}.}, we see that \eqref{eq:spec-class} is the same as the class $\kappa_\infty^\pm$ in \eqref{eq:pm-Heeg}.


Denote by $\mathfrak{m}\subset\Lambda$ the maximal ideal.

\begin{thm}[Howard]\label{thm:howard}
Let the notations and hypotheses be as in Theorem~\ref{thm:construction}.
Then both $\mathcal{S}^\pm$ and $\mathcal{X}^\pm$ have $\Lambda$-rank one, and the following divisibility holds in $\Lambda$:
\begin{equation}\label{eq:bipartite-div}
{\rm Char}_\Lambda(\mathcal{X}^\pm_{\rm tors})\supset{\rm Char}_\Lambda(\mathcal{S}^\pm/(\kappa_\infty^\pm))^2.\nonumber
\end{equation}
Moreover, if for some $j>0$ there exists $m\in\mathcal{N}_{j}^{\rm def}$ such that $\lambda_j^\pm(m)$ has non-zero image under the map
\[
\Lambda/p^j\Lambda\rightarrow\Lambda/\mathfrak{m}\Lambda\simeq\mathbf{F}_p,
\]
then the above divisibility is an equality.
\end{thm}

\begin{proof}
The element
$\kappa_\infty^\pm$ is non-torsion by Cornut--Vatsal \cite{CV-dur} (alternatively, it follows from the explicit reciprocity law of \cite[Thm.~6.2]{CW-PR} and the non-vanishing of $\mathscr{L}_\pp^{\rm BDP}$).
Suppose $j=j_0>0$ is such that the condition in the last part of the theorem holds.
Then by \cite[Lem.~3.6]{BCK}\footnote{where $\wp$ should denote the maximal ideal $\mathfrak{m}$ of $\Lambda$.} it follows that for all $j\geq j_0$ the system $\boldsymbol{\lambda}_j^\pm$ satisfies the following condition: for all height one primes $\mathfrak{P}\subset\Lambda$, the system $\boldsymbol{\lambda}_j^\pm$ contains an element with non-zero image in $\Lambda/(\mathfrak{P},p)$.
The result thus follows from \cite[Thm.~3.2.3]{howard-bipartite} with $k=k(\mathfrak{P})=1$ and the ordinary Selmer condition at the primes above $p$ replaced by the $\pm$-condition, noting that the self-duality of the latter is given by \cite[Prop.~4.11]{kim-parity}, and as shown in \cite[Lem.~6.5]{CW-PR} the analogue of the control theorem of \cite[Prop.~3.3.1]{howard-bipartite} follows from \cite[Prop.~4.18]{kim-parity}.
\end{proof}

For the proof of Theorem~\ref{thm:pm-HPMC}, we shall verify the non-vanishing condition in the last statement of Theorem~\ref{thm:howard} building on progress towards the \emph{cyclotomic} Iwasawa main conjecture.

\subsection{A consequence of the ${\rm GL}_2$-Iwasawa main conjecture}

If $g\in S_2(\Gamma_0(M))$ is any cuspidal eigenform, 
we denote by $A_g/\Q$ a ${\rm GL}_2$-type abelian variety in the isogeny class associated to $g$, and by $\mathscr{O}=\mathscr{O}_g$ the ring of integers of the completion of the Hecke field $\Q(\{a_n(g)\}_n)$ at the prime $\wp$ above $p$ determined by our fixed embedding $\iota_p:\overline{\Q}\hookrightarrow\overline{\Q}_p$.
We also let
\[
\Omega_g^{\rm cong}\in\overline{\Q}_p^\times
\]
be Hida's canonical period of $g$ as defined in \cite[\S{9.3}]{SZ}.
For any  number field $F$ and a finite prime $w$ of $F$, let $t_w(A_g/F)$ denote the Tamagawa exponent as defined in [\emph{op.\,cit.},\S{9.1}].

\begin{thm}\label{thm:IMC-ss}
Let $g\in S_2(\Gamma_0(M))$ be a cuspidal eigenform, 
and let $\wp$ be a prime of $\mathscr{O}_g$ above $p\geq 3$.
Suppose that
\begin{itemize}
\item[(i)] $\wp$ is good non-ordinary for $g$.
\item[(ii)] $M$ is square-free.
\end{itemize}
Then $L(g/K,1)$ is non-zero if and only if ${\rm Sel}_{\wp^\infty}(A_g/K)$ is finite, in which case
\[
{\rm ord}_{\wp}\biggl(\frac{L(g/K,1)}{\Omega_g^{\rm cong}}\biggr)={\rm length}_{\mathscr{O}}\,{\rm Sel}_{\wp^\infty}(A_g/K)+\sum_{w\mid M}t_w(A_g/K).
\]
\end{thm}

\begin{proof}
Let $g^K$ be the newform associated to the twist of $g$ by the quadratic character corresponding to $K$.
As a consequence of the Iwasawa Main Conjecture for ${\rm GL}_2/\Q$ for non-ordinary primes (see \cite[Thm.~C]{CCSS}, and also \cite[Cor.~1.10]{FW}) 
we have that $L(g,1)$ is non-zero if and only if ${\rm Sel}_{\wp^\infty}(A_g/\Q)$ is finite, in which case
\begin{equation}\label{eq:IMC-spec-val}
{\rm ord}_{\wp}\biggl(\frac{L(g,1)}{-2\pi i\cdot\Omega_g^+}\biggr)={\rm length}_{\mathscr{O}}\,{\rm Sel}_{\wp^\infty}(A_g/\Q)+\sum_{\ell\mid M}t_\ell(A_g/\Q),
\end{equation}
where $\Omega_g^+$ is the canonical period of $g$ (see \cite[\S{9.2}]{SZ}), and similarly\footnote{Note that \cite[Thm.~C]{CCSS} assumes square-free level as stated, but as explained in \cite[Rmk.~7.2.3]{JSW} it also applies to quadratic twists such as $g^K$.} with $g^K$ in place of $g$.
By \cite[Cor.~9.2]{SZ} we have
\[
\sum_{w\mid M}t_w(A_g/K)=\sum_{\ell\mid M}t_\ell(A_g/\Q)+\sum_{\ell\mid M}t_\ell(A_{g^K}/\Q),
\]
and since $\overline{\rho}$ is irreducible as a consequence of hypothesis (i), by Lemmas~9.5 and 9.6 in \emph{op.\,cit.} we have the period relation 
\[
\Omega_g^{\rm cong}\sim_p(2\pi i)^2\cdot\Omega_g^+\cdot\Omega_{g^K}^+.
\]
The result thus follows from the combination of \eqref{eq:IMC-spec-val} for $g$ and $g^K$.
\end{proof}

\begin{rem}\label{rem:flat/sharp}
Note that the $p$-th Fourier coefficient of the non-ordinary form $g$ in Theorem~\ref{thm:IMC-ss} is not assumed to be zero.
In fact, the result will be applied to a suitable $g$ satisfying $g\equiv f\pmod{\wp^j}$ for some $j>0$, where $f$ is as in Theorem~\ref{thm:pm-HPMC}, and so \emph{a priori} we only have  
\[
a_p(g)\equiv 0\pmod{\wp^j}.
\] 
Thus the Iwasawa theory of $g$ underlying the proof of Theorem~\ref{thm:IMC-ss} is of $\sharp/\flat$-type (after Sprung and Lei--Loeffler--Zerbes), rather than $\pm$-type (after Kobayashi and Pollack).
\end{rem}

Suppose now that $g\in S_2(\Gamma_0(M))$ is an eigenform of level $M=M^+M^-$ with $M^-$ equal to the square-free product of an \emph{odd} number of primes inert in $K$ and such that every prime factor of $M^+$ splits in $K$. 
Further, suppose that the $p$-th Fourier coefficient of $g$ satisfies $a_p(g)\equiv 0\pmod{\wp^j}$ for some $j>0$.

As explained in \cite[\S{4.3}]{CCSS} (see also \cite[Thm.~3.5]{BBL}), building on the results of \cite{ChHs1} one can associate to $g$ a pair of theta elements $\Theta_\infty^\pm(g/K)\in\mathscr{O}\dBr{\Gamma}$, and it follows from their construction and the interpolation formula in \cite[Prop.~4.3]{ChHs1} that the image of  
\begin{equation}\label{eq:Lp-pm}
L_p^\pm(g/K):=\Theta_\infty^\pm(g/K)^2
\end{equation}
under the augmentation map $\mathscr{O}\dBr{\Gamma}\rightarrow\mathscr{O}$ is equal to 
\[
\frac{L(g/K,1)}{\Omega_g^{\rm cong}}\cdot\frac{1}{\eta_{g,M^+,M^-}}\in\mathscr{O}
\] 
up to a $p$-adic unit, where $\eta_{g,M^+,M^-}\in\mathscr{O}$ is as in \cite[Eq.~(6.4)]{zhang-Kolyvagin}.
Moreover, one can easily check the implication if $g\equiv f\pmod{\wp^j}$, then
\[
\Theta_\infty^\pm(g/K)\equiv\Theta_\infty^\pm(f/K)\pmod{\wp^j\mathscr{O}\dBr{\Gamma}}
\]
(see \cite[Lem.~3.7]{BBL}).
Therefore, from the construction of the elements $\lambda_j^\pm(m)$,  
it follows that if $g$ is level-raising $f$ at $m\in\mathcal{N}_j^{\rm def}$, then the image of $\Theta_\infty^\pm(g/K)$ under the map $\mathscr{O}\dBr{\Gamma}\rightarrow\mathscr{O}\dBr{\Gamma}/\wp^j\mathscr{O}\dBr{\Gamma}$ is the same as $\lambda_j^\pm(m)$.



\subsection{Proof of Theorem~\ref{thm:pm-HPMC}}

%
By Theorem~\ref{thm:howard} and the construction of $\boldsymbol{\lambda}^\pm_j$ of Theorem~\ref{thm:construction}, it suffices to show that there exists $m\in\mathcal{N}^{\rm def}$ and an $m$-new eigenform $g\in S_2(\Gamma_0(Nm))$ with $f\equiv g\;({\rm mod}\,\wp)$, for which the $p$-adic $L$-function $L_p^\pm(g/K)$ in \eqref{eq:Lp-pm} 
is invertible.
Let
\begin{equation}\label{eq:sel-rk}
r={\rm dim}_{\mathbf{F}_p}{\rm Sel}_p(E/K).
\end{equation}
The surjectivity of $\overline{\rho}$ implies that the natural map ${\rm Sel}_{p}(E/K)\twoheadrightarrow{\rm Sel}_{p^\infty}(E/K)[p]$ is an isomorphism.
By \eqref{eq:gen-Heeg} and the $p$-parity conjecture we know that $r$ is odd, say $r=2s+1$.
By a repeated application of the argument in the proof \cite[Thm.~9.1]{zhang-Kolyvagin} (to drop the Selmer rank \eqref{eq:sel-rk} down to $1$ by adding distinct admissible primes $q_1,\dots,q_{2s}$ to the level of $f$) and the proof of Theorem~7.2 in \emph{op.\,cit.}, there exists $m=q_1\cdots q_{2s}q_r\in\mathcal{N}^{\rm def}$ and an $m$-new eigenform $g\in S_2(\Gamma_0(Nm))$ level-raising $f$ with ${\rm dim}_{\mathscr{O}/\wp}{\rm Sel}_{\wp}(A_g/K)=0$.
In particular, 
\[
{\rm Sel}_{\wp^\infty}(A_g/K)=0.
\]
From Theorem~\ref{thm:IMC-ss}, it follows that
\[
{\rm ord}_{\wp}\biggl(\frac{L(g/K,1)}{\Omega_g^{\rm cong}}\biggr)=\sum_{w\mid Nm}t_w(A_g/K).
\]
By the hypothesis that $E[p]$ is ramified at the primes $\ell\mid N^+$, we have $t_w(A_g/K)=0$ for all $w\mid N^+$, and by \cite[Thm.~6.8]{pollack-weston} (see also \cite[Thm.~6.4]{zhang-Kolyvagin}) we have 
\[
{\rm ord}_\wp(\eta_{g,N^+,N^-m})=\sum_{w\mid N^-m}t_w(A_g/K).
\] 
Therefore,
\[
\frac{L(g/K,1)}{\Omega_{g}^{\rm cong}}\frac{1}{\eta_{g,N^+,N^-m}}\in\mathscr{O}_{}^\times,
\] 
and this concludes the proof.

\bibliographystyle{amsalpha}
\bibliography{references}

\providecommand{\bysame}{\leavevmode\hbox to3em{\hrulefill}\thinspace}
\providecommand{\MR}{\relax\ifhmode\unskip\space\fi MR }
\providecommand{\MRhref}[2]{%
  \href{http://www.ams.org/mathscinet-getitem?mr=#1}{#2}
}
\providecommand{\href}[2]{#2}
\begin{thebibliography}{CGLS22}

\bibitem[AC21]{AC}
Adebisi Agboola and Francesc Castella, \emph{On anticyclotomic variants of the
  {$p$}-adic {B}irch and {S}winnerton-{D}yer conjecture}, J. Th\'{e}or. Nombres
  Bordeaux \textbf{33} (2021), no.~3, part 1, 629--658.

\bibitem[BBL24]{BBL}
Ashay Burungale, K\^az\i~m{} B\"uy\"ukboduk, and Antonio Lei,
  \emph{Anticyclotomic {I}wasawa theory of abelian varieties of {$\rm
  GL_2$}-type at non-ordinary primes}, Adv. Math. \textbf{439} (2024), Paper
  No. 109465, 63.

\bibitem[BCK21]{BCK}
Ashay Burungale, Francesc Castella, and Chan-Ho Kim, \emph{A proof of
  {P}errin-{R}iou's {H}eegner point main conjecture}, Algebra Number Theory
  \textbf{15} (2021), no.~7, 1627--1653.

\bibitem[BD95]{BD-derived-AJM}
Massimo Bertolini and Henri Darmon, \emph{Derived {$p$}-adic heights}, Amer. J.
  Math. \textbf{117} (1995), no.~6, 1517--1554.

\bibitem[BD05]{bdIMC}
M.~Bertolini and H.~Darmon, \emph{Iwasawa's main conjecture for elliptic curves
  over anticyclotomic {$\Bbb Z_p$}-extensions}, Ann. of Math. (2) \textbf{162}
  (2005), no.~1, 1--64.

\bibitem[BDP13]{bdp1}
Massimo Bertolini, Henri Darmon, and Kartik Prasanna, \emph{Generalized
  {H}eegner cycles and {$p$}-adic {R}ankin {$L$}-series}, Duke Math. J.
  \textbf{162} (2013), no.~6, 1033--1148.

\bibitem[Ben21]{benois-ht}
Denis Benois, \emph{{$p$}-adic heights and {$p$}-adic {H}odge theory}, M\'{e}m.
  Soc. Math. Fr. (N.S.) (2021), no.~167, vi + 135.

\bibitem[BLV23]{BLV}
Massimo Bertolini, Matteo Longo, and Rodolfo Venerucci, \emph{The
  anticyclotomic main conjectures for elliptic curves}, preprint, available at
  \url{https://arxiv.org/abs/2306.17784}.

\bibitem[BMC24]{burns-mc}
David Burns and Daniel Macias~Castillo, \emph{On refined conjectures of {B}irch
  and {S}winnerton-{D}yer type for {H}asse-{W}eil-{A}rtin {$L$}-series}, Mem.
  Amer. Math. Soc. \textbf{297} (2024), no.~1482, v+156.

\bibitem[Bra11]{braIMRN}
Miljan Brako{\v{c}}evi{\'c}, \emph{Anticyclotomic {$p$}-adic {$L$}-function of
  central critical {R}ankin-{S}elberg {$L$}-value}, Int. Math. Res. Not.
  \textbf{Issue 12} (2011).

\bibitem[Bur22]{bur-BDP}
Ashay Burungale, \emph{A $p$-adic {W}aldspurger {F}ormula and the {C}onjecture
  of {B}irch and {S}winnerton-{D}yer}, J. Indian Inst. Sci. (2022), 885--894.

\bibitem[Cas17]{cas-BF}
Francesc Castella, \emph{{$p$}-adic heights of {H}eegner points and
  {B}eilinson-{F}lach classes}, J. Lond. Math. Soc. (2) \textbf{96} (2017),
  no.~1, 156--180.

\bibitem[C{\c{C}}SS18]{CCSS}
Francesc Castella, Mirela {\c{C}}iperiani, Christopher Skinner, and Florian
  Sprung, \emph{On the {I}wasawa main conjectures for modular forms at
  non-ordinary primes}, available at \url{https://arxiv.org/abs/1804.10993}.

\bibitem[CGLS22]{CGLS}
Francesc Castella, Giada Grossi, Jaehoon Lee, and Christopher Skinner, \emph{On
  the anticyclotomic {I}wasawa theory of rational elliptic curves at
  {E}isenstein primes}, Invent. Math. \textbf{227} (2022), no.~2, 517--580.

\bibitem[CGS23]{CGS}
Francesc Castella, Giada Grossi, and Christopher Skinner, \emph{Mazur's main
  conjecture at {E}isenstein primes}, available at
  \url{https://arxiv.org/abs/2303.04373}.

\bibitem[CH18a]{cas-hsieh1}
Francesc Castella and Ming-Lun Hsieh, \emph{Heegner cycles and p-adic
  {L}-functions}, Math. Ann. \textbf{370} (2018), no.~1-2, 567--628.

\bibitem[CH18b]{ChHs1}
Masataka Chida and Ming-Lun Hsieh, \emph{Special values of anticyclotomic
  {$L$}-functions for modular forms}, J. Reine Angew. Math. \textbf{741}
  (2018), 87--131.

\bibitem[CLW22]{CLW}
Francesc Castella, Zheng Liu, and Xin Wan, \emph{Iwasawa-{G}reenberg main
  conjecture for nonordinary modular forms and {E}isenstein congruences on
  {GU}(3,1)}, Forum Math. Sigma \textbf{10} (2022), Paper No. e110, 90.

\bibitem[Col00]{colmez-Lp}
Pierre Colmez, \emph{Fonctions {$L$} {$p$}-adiques}, no. 266, 2000,
  S\'{e}minaire Bourbaki, Vol. 1998/99, pp.~Exp. No. 851, 3, 21--58.

\bibitem[CV07]{CV-dur}
Christophe Cornut and Vinayak Vatsal, \emph{Nontriviality of {R}ankin-{S}elberg
  {$L$}-functions and {CM} points}, {$L$}-functions and {G}alois
  representations, London Math. Soc. Lecture Note Ser., vol. 320, Cambridge
  Univ. Press, Cambridge, 2007, pp.~121--186.

\bibitem[CW24]{CW-PR}
Francesc Castella and Xin Wan, \emph{Perrin-{R}iou's main conjecture for
  elliptic curves at supersingular primes}, Math. Ann. \textbf{389} (2024),
  no.~3, 2595--2636.

\bibitem[DI08]{darmon-iovita}
Henri Darmon and Adrian Iovita, \emph{The anticyclotomic main conjecture for
  elliptic curves at supersingular primes}, J. Inst. Math. Jussieu \textbf{7}
  (2008), no.~2, 291--325.

\bibitem[DT94a]{DT-dmj}
Fred Diamond and Richard Taylor, \emph{Lifting modular mod {$l$}
  representations}, Duke Math. J. \textbf{74} (1994), no.~2, 253--269.

\bibitem[DT94b]{DT-inv}
\bysame, \emph{Nonoptimal levels of mod {$l$} modular representations}, Invent.
  Math. \textbf{115} (1994), no.~3, 435--462.

\bibitem[Edi97]{edix}
Bas Edixhoven, \emph{Serre's conjecture}, Modular forms and {F}ermat's last
  theorem ({B}oston, {MA}, 1995), Springer, New York, 1997, pp.~209--242.

\bibitem[FW21]{FW}
Olivier Fouquet and Xin Wan, \emph{On the {I}wasawa {M}ain {C}onjecture for
  universal families of modular motives}, available at
  \url{https://arxiv.org/abs/2107.13726}.

\bibitem[Gre99]{greenberg-cetraro}
Ralph Greenberg, \emph{Iwasawa theory for elliptic curves}, Arithmetic theory
  of elliptic curves ({C}etraro, 1997), Lecture Notes in Math., vol. 1716,
  Springer, Berlin, 1999, pp.~51--144.

\bibitem[Gre16]{greenberg-str-Sel}
\bysame, \emph{On the structure of {S}elmer groups}, Elliptic curves, modular
  forms and {I}wasawa theory, Springer Proc. Math. Stat., vol. 188, Springer,
  Cham, 2016, pp.~225--252.

\bibitem[GZ86]{GZ}
Benedict~H. Gross and Don~B. Zagier, \emph{Heegner points and derivatives of
  {$L$}-series}, Invent. Math. \textbf{84} (1986), no.~2, 225--320.

\bibitem[HL19]{hatley-lei-MRL}
Jeffrey Hatley and Antonio Lei, \emph{Comparing anticyclotomic {S}elmer groups
  of positive coranks for congruent modular forms}, Math. Res. Lett.
  \textbf{26} (2019), no.~4, 1115--1144.

\bibitem[How04]{howard-derived}
Benjamin Howard, \emph{Derived {$p$}-adic heights and {$p$}-adic
  {$L$}-functions}, Amer. J. Math. \textbf{126} (2004), no.~6, 1315--1340.

\bibitem[How06]{howard-bipartite}
\bysame, \emph{Bipartite {E}uler systems}, J. Reine Angew. Math. \textbf{597}
  (2006), 1--25.

\bibitem[JSW17]{JSW}
Dimitar Jetchev, Christopher Skinner, and Xin Wan, \emph{The {B}irch and
  {S}winnerton-{D}yer formula for elliptic curves of analytic rank one}, Camb.
  J. Math. \textbf{5} (2017), no.~3, 369--434.

\bibitem[Kim07]{kim-parity}
Byoung~Du Kim, \emph{The parity conjecture for elliptic curves at supersingular
  reduction primes}, Compos. Math. \textbf{143} (2007), no.~1, 47--72.

\bibitem[KO20]{KO}
Shinichi Kobayashi and Kazuto Ota, \emph{Anticyclotomic main conjecture for
  modular forms and integral {P}errin-{R}iou twists}, Development of {I}wasawa
  theory---the centennial of {K}. {I}wasawa's birth, Adv. Stud. Pure Math.,
  vol.~86, Math. Soc. Japan, Tokyo, [2020] \copyright 2020, pp.~537--594.

\bibitem[Kol88]{kol88}
V.~A. Kolyvagin, \emph{Finiteness of {$E({\bf Q})$} and {S}ha(${E}$,${{\bf
  Q}}$) for a subclass of {W}eil curves}, Izv. Akad. Nauk SSSR Ser. Mat.
  \textbf{52} (1988), no.~3, 522--540, 670--671.

\bibitem[Maz78]{mazur-prime}
B.~Mazur, \emph{Rational isogenies of prime degree (with an appendix by {D}.
  {G}oldfeld)}, Invent. Math. \textbf{44} (1978), no.~2, 129--162.

\bibitem[MR04]{MR-KS}
Barry Mazur and Karl Rubin, \emph{Kolyvagin systems}, Mem. Amer. Math. Soc.
  \textbf{168} (2004), no.~799, viii+96. \MR{2031496 (2005b:11179)}

\bibitem[Nek95]{nekovar302}
Jan Nekov{\'a}{\v{r}}, \emph{On the {$p$}-adic height of {H}eegner cycles},
  Math. Ann. \textbf{302} (1995), no.~4, 609--686. \MR{1343644 (96f:11073)}

\bibitem[Nek01]{nekovarII}
\bysame, \emph{On the parity of ranks of {S}elmer groups. {II}}, C. R. Acad.
  Sci. Paris S\'er. I Math. \textbf{332} (2001), no.~2, 99--104.

\bibitem[Nek06]{nekovar310}
\bysame, \emph{Selmer complexes}, Ast\'erisque (2006), no.~310, viii+559.

\bibitem[PR87]{PR-HP}
Bernadette Perrin-Riou, \emph{Fonctions {$L$} {$p$}-adiques, th\'eorie
  d'{I}wasawa et points de {H}eegner}, Bull. Soc. Math. France \textbf{115}
  (1987), no.~4, 399--456.

\bibitem[PR92]{PR-ht}
\bysame, \emph{Th\'{e}orie d'{I}wasawa et hauteurs {$p$}-adiques}, Invent.
  Math. \textbf{109} (1992), no.~1, 137--185.

\bibitem[PR93]{PR-ht-abvar}
\bysame, \emph{Th\'{e}orie d'{I}wasawa et hauteurs {$p$}-adiques (cas des
  vari\'{e}t\'{e}s ab\'{e}liennes)}, S\'{e}minaire de {T}h\'{e}orie des
  {N}ombres, {P}aris, 1990--91, Progr. Math., vol. 108, Birkh\"{a}user Boston,
  Boston, MA, 1993, pp.~203--220.

\bibitem[PW11]{pollack-weston}
Robert Pollack and Tom Weston, \emph{On anticyclotomic {$\mu$}-invariants of
  modular forms}, Compos. Math. \textbf{147} (2011), no.~5, 1353--1381.

\bibitem[Rib84]{ribet-ICM}
Kenneth~A. Ribet, \emph{Congruence relations between modular forms},
  Proceedings of the {I}nternational {C}ongress of {M}athematicians, {V}ol.\ 1,
  2 ({W}arsaw, 1983), PWN, Warsaw, 1984, pp.~503--514.

\bibitem[San23]{sano}
Takamichi Sano, \emph{Derived {B}ockstein regulators and anticyclotomic
  $p$-adic {B}irch and {S}winnerton-{D}yer conjectures}, preprint.

\bibitem[Ski20]{skinner}
Christopher Skinner, \emph{A converse to a theorem of {G}ross, {Z}agier, and
  {K}olyvagin}, Ann. of Math. (2) \textbf{191} (2020), no.~2, 329--354.

\bibitem[SZ14]{SZ}
Christopher Skinner and Wei Zhang, \emph{Indivisibility of {H}eegner points in
  the multiplicative case}, available at \url{https://arxiv.org/abs/1407.1099}.

\bibitem[Zha14]{zhang-Kolyvagin}
Wei Zhang, \emph{Selmer groups and the indivisibility of {H}eegner points},
  Camb. J. Math. \textbf{2} (2014), no.~2, 191--253.

\end{thebibliography}
\end{document}